\documentclass{amsart}
\usepackage{mathabx}
\usepackage{amsmath,amscd,amsthm}
\usepackage{amsfonts,amssymb}
\usepackage[all]{xy}
\usepackage{mathtools}
\usepackage{enumerate}

\usepackage{hyperref}

\usepackage[usenames, dvipsnames]{color}

\definecolor{NTNUblue}{RGB}{0,80,158}
\definecolor{NTNUbluesupport}{RGB}{62,98,138}
\definecolor{NTNUorange}{RGB}{239,129,20}

\newcommand{\F}{\mathbb{F}}

\newcommand{\R}{\mathbb{R}}
\newcommand{\Z}{\mathbb{Z}}

\newcommand{\BB}{\mathbf{B}} 

\newcommand{\SB}{\mathbf{S}}

\newcommand{\pp}{\mathfrak{p}} 

\newcommand{\frp}{\mathfrak{p}}

\newcommand{\Rh}{\mathcal{R}}

\newcommand{\Hom}{\mathrm{Hom}}

\newcommand{\Inv}{\mathrm{Inv}}

\newcommand{\bbb}{\bullet}

\newcommand{\Cb}{C^\bbb}

\newcommand{\Hb}{H^\bbb}

\newcommand{\Bb}{B^\bbb}

\newcommand{\Db}{D^\bbb}

\newcommand{\Spec}{\mathrm{Spec}}

\numberwithin{equation}{subsection}
\theoremstyle{plain}

\newtheorem{thm}[subsection]{Theorem}
\newtheorem{prop}[subsection]{Proposition}
\newtheorem{lemma}[subsection]{Lemma}
\newtheorem{cor}[subsection]{Corollary}
\theoremstyle{definition}
\newtheorem{defn}[subsection]{Definition}
\newtheorem{example}[subsection]{Example}

\newtheorem{notn}[subsection]{Notation}

\theoremstyle{remark}
\newtheorem{rem}[subsection]{Remark}
\newtheorem{remark}[subsection]{Remark}
\newtheorem{rems}[subsection]{Remarks}

\input cyracc.def

\begin{document}

\title{Quasi-boolean groups}

\author{Ambrus P\'al}
\address{Department of Mathematics, 180 Queen's Gate, Imperial College, London, SW7 2AZ, United Kingdom}
\email{a.pal@imperial.ac.uk}
\author{Gereon Quick}
\address{Department of Mathematical Sciences, NTNU, NO-7491 Trondheim, Norway}
\email{gereon.quick@ntnu.no}
\thanks{Both authors were partially supported by RCN Project No.\,313472 {\it Equations in Motivic Homotopy}, 
and the project \emph{Pure Mathematics in Norway} funded by the Trond Mohn Foundation.}

\date{}

\begin{abstract} 
We give several equivalent characterisations of the maximal pro-$2$ quotients of real projective groups. 
In particular, for pro-$2$ real projective groups we provide a presentation in terms of generators and relations, and a purely cohomological characterisation. 
As a consequence we explicitly reconstruct such groups from their mod $2$ cohomology rings. 
\end{abstract}
\maketitle

\section{Introduction}

Let $G$ be a profinite group. An {\it embedding problem for $G$} is a solid diagram:
\begin{align*}
\xymatrix{
 & G \ar[d]^-{\phi}\ar@{.>}[ld]_-{\widetilde{\phi}}
\\
B\ar[r]_-{\alpha} & A}
\end{align*}
where $A$ and $B$ are finite groups, the solid arrows are continuous homomorphisms and $\alpha$ is surjective. 
A {\it solution} of an embedding problem is a continuous homomorphism $\widetilde{\phi} \colon G\to B$ which makes the diagram commutative.  
We say that the embedding problem above is {\it real} if for every involution $t\in G$ with $\phi(t)\neq 1$ there is an involution $b \in B$ with $\alpha(b)=\phi(t)$, 
i.e., if involutions do not provide an obstruction for the existence of a solution. 

\begin{defn}\label{def:real_projective_group}
Following Haran and Jarden \cite[page 450]{HJ} we say that a profinite group $G$ is {\it real projective} 
if the subset $\Inv(G)$ of involutions is closed in $G$ 
and if every real embedding problem for $G$ has a solution. 
\end{defn}

\begin{remark}\label{rem:equiv_of_defns}
By \cite[Remark 7.6 and Proposition 7.7 on page 473]{HJ}, 
the condition that $\Inv(G)$ is closed in $G$ is equivalent to the condition that $G$ 
has an open subgroup without $2$-torsion which is used in \cite[Definition 1.1]{PQ}. 
To see that these conditions are equivlaent, let $g \in G$ be an element which belongs to the closure of $\Inv(G)$ satisfies $g^2=1$. 
Then $g \in \Inv(G)$ or $g=1$. 
Thus, $\Inv(G)$ is closed in $G$ if and only if $G$ has an open 
subgroup $U$ 
such that $U \cap \Inv(G) = \emptyset$. 
\end{remark}


By the work of Haran--Jarden \cite{HJ}, real projective groups play an important role in Galois theory as they are the absolute Galois groups of pseudo real closed fields.     
In fact, it follows from the work of Haran \cite{Ha} and Haran--Jarden \cite{HJ} 
that the absolute Galois groups of fields with virtual cohomological dimension at most one, which is a slightly larger class of fields, 
are real projective (see \cite[Corollary 2.4]{PQ}).  
For a real projective group $G$, we show in \cite[Theorem 1.3]{PQ} that the differential graded algebra $\Cb(G,\F_2)$ of continuous cochains is formal, 
i.e., $\Cb(G,\F_2)$ is quasi-isomorphic as differential graded algebras to its cohomology algebra. 
Roughly speaking, this means that the cohomology algebra of a real projective group already contains all the information of the differential graded algebra $\Cb(G,\F_2)$. 
The purpose of the present paper is to show that, in fact, the maximal pro-$2$ quotient of a real projective group can be reconstructed entirely from the mod $2$ cohomology ring. 
In particular, we show that pro-$2$ real projective groups can be reconstructed entirely from their mod $2$ cohomology ring. 
In order to further describe our results we recall the following terminology from \cite{PQ}. 

\begin{defn}\label{def:boolean_dual_sum_intro} 
We call an $\F_2$-algebra $\Bb = \bigoplus_{i\ge 0} B^i$ a {\it graded Boolean algebra} if $B^0=\mathbb F_2$ and 
there is a Boolean ring $B$ (see Section \ref{sec:Boolean_rings}) 
such that, for every $i \geq 1$, we have $B^i=B$ and multiplication in $\Bb$ is induced by $B$. 
We call an $\F_2$-algebra $\Db = \bigoplus_{i\ge 0} D^i$ a {\it dual algebra} if $D^0=\F_2$, and $D^i=0$ for $i \geq 2$. 
The {\it connected sum} $\Db \sqcap \Bb$ is the graded $\F_2$-algebra with $(\Db \sqcap \Bb)^0=\F_2$, $(\Db \sqcap \Bb)^i= D^i \oplus B^i$ for $i\geq 1$ and multiplication $D^1B^i$ and $B^iD^1$ is set to be zero for all $i \geq 1$. 
\end{defn}

In \cite{PQ}, we deduce from Scheiderer's work in \cite{Schbook} that the mod $2$ cohomology algebra of a real projective group is a connected sum of a dual and a Boolean graded algebra. 
Hence, in the terminology of Definition \ref{def:Boolean_etc_intro} below, the maximal pro-$2$ quotient of a real projective group is a cohomologically quasi-Boolean group. 
The anonymous referee of \cite{PQ} suggested that the latter property may even characterise pro-$2$ real projective groups completely. 
The purpose of the present paper is to prove this conjecture. 
In fact, we prove the stronger fact that every pro-$2$ real projective group has an explicit description as a certain free pro-$2$ product with explicit generators and relations provided by the cohomology ring (see Theorem \ref{reconstruct} below).  
This result significantly strengthens a consequence of the main result of \cite{PQ}. 
According to the latter, for a real projective group $G$, 
the differential graded $\F_2$-algebra $\Cb(G,\F_2)$  and its cohomology $\Hb(G,\F_2)$ is Koszul. 
Hence, by a theorem of Positselski, the $\F_2$-linear co-algebra of the maximal pro-$2$ quotient $H$ 
of $G$ can be reconstructed explicitly from the cohomology ring $\Hb(H,\F_2)\cong \Hb(G,\F_2)$ using the bar construction (see Example 6.3 of \cite{Po}). 
Here we reconstruct the group itself via an even more transparent recipe. 
\\

We now describe our main results in more detail. 
To do so we need the following constructions and terminology. 
First we recall free products. 

\begin{defn} 
Let $G_1$ and $G_2$ be two pro-$p$ groups. 
Let $G_1*G_2$ be the discrete free product of $G_1$ and $G_2$ and let $\mathcal N$ be the family of normal subgroups $N$ of $G$ such that $(G_1*G_2)/N$ is a finite $p$-group and $N \cap G_1, N\cap G_2$ are open subgroups of $G_1, G_2$ respectively. 
Then
\[
G_1 *_pG_2 \coloneqq \varprojlim_{N\in\mathcal N}(G_1*G_2)/N
\]
is called the {\it free pro-$p$ product} of $G_1$ and $G_2$. 
\end{defn}

We note that the free pro-$p$ product of $G_1$ and $G_2$ satisfies the usual universal property 
with respect to continuous homomorphisms from $G_1$ and $G_2$ to pro-$p$ groups 
(see e.g.~\cite[Proposition 3.1.1]{HJbook}).  
Next we are going to define free products over topological spaces where we refer to 
\cite{GR}, \cite{H1}, \cite{H2}, \cite{Ha}, \cite{HJ2}, \cite[Chapter 4]{HJbook}, \cite{Melnikov}, \cite[Chapter 5]{Ribesbook} 
for properties and variations of free products of profinite groups. 

\begin{defn}\label{def:B_of_X}
For every topological space $X$, 
let $\bigast_X\mathbb Z/2\mathbb Z$ denote the group which is freely generated by the elements of $X$, 
subject to the relation that these elements are involutions, 
and let $\mathcal N$ be the family of normal subgroups $N$ of $\bigast_X\mathbb Z/2\mathbb Z$ 
such that $(\bigast_X\Z/2\Z)/N$ is a finite $2$-group and the composition of the natural inclusion 
$\iota \colon X\to \bigast_X\Z/2\Z$ and the quotient homomorphism
$\bigast_X\mathbb Z/2\Z \to (\bigast_X\Z/2\Z)/N$ is continuous 
with respect to the discrete topology on $(\bigast_X\Z/2\Z)/N$. 
Then we set 
\[
\mathbb B(X) \coloneqq \varprojlim_{N\in\mathcal N}(\bigast_X\Z/2\Z)/N
\]
and refer to $\mathbb B(X)$ as the {\it free pro-$2$ product of order two groups over $X$}. 
\end{defn}


Finally, we are going to use the following terminology for pro-$2$ groups, 
where we refer to \cite[1.5 on pages 7--8]{Se} for the notion of a free pro-$2$ group. 

\begin{defn}\label{def:Boolean_etc_intro}
We say that a pro-$2$ group is a {\it Boolean group} if it is isomorphic to $\mathbb B(X)$ for some topological space $X$. 
We say that a pro-$2$ group is a {\it quasi-Boolean group} if it is the free product of a free pro-$2$ group 
and a Boolean group. 
We say that a pro-$2$ group is a {\it cohomologically Boolean group} if its mod $2$ cohomology is a graded Boolean algebra. 
We say that a pro-$2$ group is a {\it cohomologically quasi-Boolean group} if its mod $2$ cohomology is the connected sum of a dual algebra and a graded Boolean algebra.
\end{defn}

For a Boolean group $\mathbb B(X)$, we will see in Proposition \ref{replace} that we may assume without the loss of generality that $X$ is profinite. 
We note that a Boolean group is the free product over 
a constant sheaf of a group of order $2$ and a profinite space $X$ in the sense of \cite[Chapter 4]{HJbook},  
and quasi-Boolean groups are the real free groups in the category of pro-$2$ groups in the sense of \cite{H2} and \cite{HJ2}.
We can now state our main results. 

\begin{thm}\label{bigbigbig} Let $G$ be a pro-$2$ group. Then the following are equivalent:
\begin{enumerate}
\item[$(i)$] $G$ is quasi-Boolean. 
\item[$(ii)$] $G$ is real projective. 
\item[$(iii)$] $G$ is the maximal pro-$2$ quotient of a real projective profinite group. 
\item[$(iv)$] $G$ is cohomologically quasi-Boolean. 
\end{enumerate}
\end{thm}

We note that the equivalence of $(ii)$ and $(iii)$ follows from the arguments developed in \cite[Section 3]{HJ2}, 
and the equivalence of $(i)$ and $(ii)$ has already been proven in \cite[Corollary 3.3]{HJ2} and \cite[Proposition 4.2]{H2} 
(assuming that $X$ is a closed system of representatives of the conjugacy classes of involutions in $G$). 
We give an alternative proof of the equivalence of $(i)$ and $(ii)$ with no assumption on $X$ using results of \cite{HJ2} 
since this equivalence will follow from results we need to develop for the remaining equivalence with $(iv)$. 
Hence the most interesting feature of Theorem \ref{bigbigbig} 
beyond the results of \cite{H2} and \cite{HJ2} 
is that it incorporates a purely cohomological characterisation of these pro-$2$ groups. 
Our proof is a bit more involved than it might be anticipated.  
The proof uses results of Haran--Jarden in the arithmetic of fields, 
for example a group-theoretical characterisation of real projective groups, 
a theorem on the existence of sections of profinite principle $G$-bundles, 
and profinite versions of two theorems of Quillen on group cohomology. 
We will also show the following: 

\begin{cor}\label{justboole} Let $G$ be a pro-$2$ group. Then the following are equivalent:
\begin{enumerate}
\item[$(i)$] $G$ is Boolean. 
\item[$(ii)$] $G$ is cohomologically Boolean. 
\end{enumerate}
\end{cor}

Using our results we then demonstrate that quasi-Boolean groups can, in fact, 
be reconstructed from their mod $2$ cohomology. 
For every set $Y$, let $F(Y)$ denote the free pro-$2$ group as defined in \cite[Section 1.5 on pages 7--8]{Se}. 

\begin{thm}\label{reconstruct} 
Let $G$ be a quasi-Boolean pro-$2$ group 
such that $\Hb(G,\F_2)$ is the connected sum of a dual algebra $\Db$ 
and a graded Boolean algebra $\Bb$ associated to the Boolean ring $B$. 
Let $Y$ be a basis of $D^1$ and let $X$ be the spectrum of $B$. 
Then $G$ is isomorphic to $F(Y)*_2\mathbb B(X)$. 
\end{thm}


\subsection*{Content}
In Section \ref{sec:Boolean_rings} we give a modern exposition of the theory of Boolean rings, including Stone duality, using now standard tools from commutative algebra. 
We cover some background material on profinite spaces, including the profinite completion functor, in the Section \ref{sec:profinite_spaces}. 
In Section \ref{sec:prof_G_bundles} we prove that profinite principal $G$-bundles have sections, a result originally announced by Morel in \cite{Mo} without proof. 
We prove that the class of pro-$2$ real projective groups and the maximal pro-$2$ quotients of real projective groups are the same in Section \ref{sec:max_2_quotients}, 
the key step being a simple group-theoretical lemma. 
In Section \ref{sec:pro2_real_vs_quasi_Boolean} we show that the class of quasi-Boolean and pro-$2$ real projective groups are the same, 
heavily relying on the main results of the previous sections and several theorems of \cite{HJ2}. 
Then, in Section \ref{sec:Quillen_consequences}, we present profinite versions of some classical results of Quillen on the cohomology of groups, 
partially following the suggestion at the end of \cite{Sc}, 
and use these results to derive a local-global principle for the cohomology of cohomologically quasi-Boolean groups 
analogous to Scheiderer's theorem for real projective groups. 
In Section \ref{sec:main_thm_and_consequences} we prove the main results using cohomological obstruction theory for central embedding problems and the local-global principle of the previous section.

\subsection*{Acknowledgement} 
This paper was written as a response to one of the questions asked by the referee to our paper \cite{PQ}. 
We are very thankful for the anonymous referee's comments. 
We thank the anonymous referee of the present paper for many comments and suggestions that helped to improve the manuscript. 

%

\section{A primer on Boolean rings}\label{sec:Boolean_rings}

In this section we recall and prove the results we need on Boolean rings. 
In particular, we deduce Stone duality in Theorem \ref{duality}. 
We note that the content of this section is very similar to the content of \cite[Appendix A]{PQ}. 
We include it here for convenience of the reader. 

\begin{defn} 
A ring $R$ is called {\it Boolean} if $x^2=x$ for every $x\in R$.
\end{defn}

\begin{example}\label{ex1.2} 
The field with two elements $\mathbb F_2$ is a Boolean ring. 
In fact, since $x(1-x)=0$ for all $x$ in a Boolean ring, $\F_2$ is the only Boolean integral domain. 
The direct product of Boolean rings is Boolean, and so, for every set $X$, the direct product ring:
\[
\mathbb F_2^X=\prod_{i\in X}\mathbb F_2
\]
is a Boolean ring. 
Now let $X$ be a topological space, and let $\BB(X)$ denote the ring of functions $f \colon X\to\mathbb F_2$ which are continuous with respect to the discrete topology on $\mathbb F_2$. Since the subrings of Boolean rings are Boolean, and $\BB(X)$ is a subring of $\mathbb F_2^X$, we get that $\BB(X)$ is Boolean, too. 
\end{example}

\begin{prop}\label{basic} 
In a Boolean ring $R$ the following hold:
\begin{enumerate}
\item[$(i)$]  we have $2x=0$ for every $x\in R$; 
\item[$(ii)$]  every prime ideal $\mathfrak p$ is maximal, and $R/\mathfrak p$ is the field with two elements; 
\item[$(iii)$] we have $(x,y)=(x+y-xy)$ for every $x,y\in R$; 
\item[$(iv)$] every finitely generated ideal is principal.
\end{enumerate}
\end{prop}
\begin{proof} 
Since $2x=(2x)^2=4x^2=4x$,  
we get that $2x=0$ by subtracting $2x$ from both sides. Now let
$\mathfrak p$ be a prime ideal in $R$. Then the quotient $R/\mathfrak p$ is a Boolean ring.  
For every $x\in R/\mathfrak p$, we have $x(1-x)=0$ which implies that $x=0$ or $x=1$ since $R/\mathfrak p$ is an integral domain. Claim $(ii)$ follows. Note that
\[
x(x+y-xy)=x^2+xy-x^2y=x+xy-xy=x.
\]
Hence $x,y\in(x+y-xy)$. 
Since $x+y-xy\in(x,y)$, claim $(iii)$ is clear. Let $I=( x_1,x_2,\ldots,x_n)$ be a finitely generated ideal of $R$. 
Since
\[
I=((x_1,x_2,\ldots,x_{n-1}),x_n),
\]
we may assume by induction on $n$ that $I=(x,y)$ for some $x,y\in R$. 
Claim $(iv)$ now follows from part $(iii)$. 
\end{proof}

\begin{defn}
A topological space $X$ is called {\it totally separated} if 
for all $x \ne y$ in $X$ there exists a closed and open set $U \subset X$ such that $x\in U$ and $y \notin U$. 
Equivalently, $X$ is totally separated if for any two distinct points 
$x,y\in X$ there exist disjoint open sets $U\subset X$ containing $x$ and $V\subset X$ containing $y$ such that $X$ is the union of 
$U$ and $V$. 
\end{defn}

\begin{prop}\label{spec} 
The spectrum $\mathrm{Spec}(R)$ of a Boolean ring with the Zariski topology is compact and totally separated.
\end{prop}
\begin{proof} 
By \cite[Proposition I.1.1.10]{EGA1}, 
the spectrum of a commutative ring with a unity with the Zariski topology is compact. 
Thus, $\mathrm{Spec}(R)$ is compact.  
Now let $\mathfrak p,\mathfrak q\in\mathrm{Spec}(R)$ be two distinct points. 
Since they are maximal ideals by part $(ii)$ of Proposition \ref{basic}, 
there is an $x\in R$ such that $x\in\mathfrak p$ and $x\not\in\mathfrak q$. 
Since $R/\mathfrak p=\mathbb F_2$ by part $(ii)$ of Proposition \ref{basic}, the former is equivalent to $1-x\not\in\mathfrak p$. 
As usual, for every $f\in R$, 
let $D(f)\subseteq\mathrm{Spec}(R)$ denote the open subset 
\begin{align}\label{eq:def_of_Df}
D(f)=\{\mathfrak p\in\mathrm{Spec}(R)\mid f\not\in\mathfrak p\}.
\end{align}
Then $\mathfrak p\in D(1-x),\mathfrak q\in D(x)$, the intersection $D(x)\cap D(1-x)$ is empty by part $(ii)$ of Proposition \ref{basic}, while the union $D(x)\cup D(1-x)$ is $\mathrm{Spec}(R)$, since if $x,1-x\in\mathfrak m$ for some $\mathfrak m\in\mathrm{Spec}(R)$ then $1\in\mathfrak m$ which is a contradiction. 
\end{proof}

\begin{notn}\label{notn:sigma_map} 
Let $R$ be a Boolean ring. 
Then for every $a\in R$ the {\it corresponding section of the structure sheaf} 
of $\mathrm{Spec}(R)$ is the function
\[
\sigma(a) \colon \mathrm{Spec}(R) \to \mathbb F_2
\]
defined by $\sigma(a)(\pp) = 0$ if $a+ \pp = \pp$, i.e., $a\in \pp$, and $\sigma(a)(\pp) = 1$ if $a \notin \pp$.  
For every $a \in R$, $\sigma(a)$ is continuous
since $\{\frp \in \Spec(R) ~|~ \sigma(a)(\frp)=1\}=D(a)$
and $\{\frp \in \Spec(R) ~|~ \sigma(a)(\frp)=0\}=D(1-a)$ are open,
because $a\in\frp$ if and only if $1-a\notin \frp$,
by part $(ii)$ of Proposition \ref{basic}. 
The induced map
\[
\sigma \colon R \to \BB(\mathrm{Spec}(R))
\]
is a ring homomorphism. 
\end{notn}

\begin{thm}\label{sigma} 
For every Boolean ring $R$, the map
$\sigma \colon R \to \BB(\mathrm{Spec}(R))$ is an isomorphism. 
\end{thm}
\begin{proof} 
For every $x\in R$, we have $x^n=x$ by induction, so if $x$ is nilpotent, then it is zero. 
Therefore the nilradical of $R$ is zero. 
By \cite[Proposition 1.8 on page 5]{AM}, this shows that the intersection of all prime ideals in $R$ is zero. 
Since $\sigma(a)$ is the zero map if and only if $a$ lies in the intersection of all prime ideals by definition, 
this shows that $\sigma$ is injective. 
It remains to show that $\sigma$ is surjective. 
Let $f \colon \mathrm{Spec}(R)\to\mathbb F_2$ be a continuous function. 
Then the set
\[
D(f)=\{\mathfrak p\in\mathrm{Spec}(R)\mid f(\mathfrak p)=1\}
\]
is a closed subset. 
Thus, $D(f)$is compact since $\Spec(R)$ is compact by Proposition \ref{spec}. 
But $D(f)$ is also open, so it can be covered by open subsets of the form $D(a)$, 
where $a\in R$ according to the Zariski topology. 
Since $D(f)$ is compact, it can be covered by finitely many such, i.e., 
\[
D(f)=D(a_1)\cup D(a_2)\cup\cdots\cup D(a_n)
\]
for some $a_1,a_2,\ldots,a_n\in R$. 
By part $(iv)$ of Proposition \ref{basic}, there is an $a\in R$ such that $(a)=(a_1,\ldots,a_n)$. 
Then
\begin{align*}
D(a)&=\{\mathfrak p\in\mathrm{Spec}(R)\mid a\not\in\mathfrak p\} \\
& = \{\mathfrak p\in\mathrm{Spec}(R)\mid (a)\not\subseteq\mathfrak p\} \\
& = \{\mathfrak p\in\mathrm{Spec}(R)\mid (a_1,\ldots,a_n)\not\subseteq
\mathfrak p\} \\
& = \{\mathfrak p\in\mathrm{Spec}(R)\mid a_i\not\in
\mathfrak p\textrm{ for some $i$}\} \\
& = D(a_1)\cup D(a_2)\cup\cdots\cup D(a_n),\end{align*}
so $D(f) = D(a) = D(\sigma(a))$. 
Since $f$ and $\sigma(a)$ take values in $\mathbb F_2$, we get that $f = \sigma(a)$. 
\end{proof}

\begin{notn}\label{def1.7} 
Let $X$ be a topological space. 
Then for every $p\in X$ the set
\[
\beta(p)=\{ x \in \BB(X)\mid x(p)=0\}
\]
is the kernel of a surjective ring homomorphism $\BB(X)\to\mathbb F_2$, so it is a maximal ideal in $\BB(X)$. 
Consequently, we have an induced map
\[
\beta \colon X\to\mathrm{Spec}(\BB(X)).
\]
For every $x \in \BB(X)$, we get from \eqref{eq:def_of_Df} for $R=\BB(X)$ and $f=x$ an equality of sets 
\[
\beta^{-1}(D(x)) = \{p\in X\mid x(p)=1\}.
\]
Hence the preimage $\beta^{-1}(D(x))$ is open. 
Since the sets $\{D(x)\mid x \in \BB(X)\}$ form a sub-basis of
$\mathrm{Spec}(\BB(X))$, we get that 
$\beta \colon X\to \mathrm{Spec}(\BB(X))$ is continuous.  
\end{notn}

\begin{thm}\label{beta} 
When $X$ is compact and totally separated, then
$\beta$ is a homeomorphism. 
\end{thm}

\begin{proof} Since $X$ is compact and $\mathrm{Spec}(\BB(X))$ is Hausdorff, 
it will be sufficient to show that $\beta$ is a bijection by \cite[I \S 9.4, Corollary 2 on page 87]{BourbakiTop}, 
which states that a continuous map from a compact space to a Hausdorff space is a homeomorphism if it is a bijection.  
Let $x,y\in X$ be two distinct points. 
Since $X$ is totally separated, 
there exist disjoint open sets $U\subset X$ containing $x$ and $V\subset X$ containing $y$ such that $X$ is the union of $U$ and $V$. 
Let $f \colon X\to\mathbb F_2$ be the characteristic function of $U$. 
Thus $f(x)=1$ if $x\in U$ and $f(x)=0$ if $x\in
X\setminus U=V$. 
Hence $f$ is continuous and an element in $\BB(X)$.  
Clearly $f\in\beta(y)$, but $f\not\in\beta(x)$, so $\beta$ is injective.

For every ideal $I \triangleleft \BB(X)$, 
let $Z(I)\subseteq X$ denote the closed subset
\[
Z(I)=\{x\in X\mid f(x)=0\quad(\forall f\in I)\}.
\]
We claim that for every proper ideal $I\triangleleft\BB(X)$ the set $Z(I)$ is non-empty. First consider the case when $I=(f)$ for some $f \in \BB(X)$. Then
\[
Z(I)=\{x\in X\mid f(x)=0\},
\]
so if this set is empty, then $f$ is the identically one function, and hence $I=(1)=\BB(X)$ is not proper, a contradiction. 
Next consider the case when $I$ is finitely generated. 
Then $I$ is principal by part $(iv)$ of Proposition \ref{basic}, 
so $Z(I)$ is non-empty by the above. 
Finally, consider the general case. 
Then $Z(I)$ is the intersection of sets of the form $Z(J)$ where $J$ is a finitely generated ideal of $I$. Since the latter collection of sets is closed under finite intersections, and each member is non-empty by the above, we get $Z(I)$ is also non-empty, since $X$ is compact.

Now let $\mathfrak m\triangleleft\BB(X)$ be a maximal ideal. 
By the claim in the previous paragraph, there is an $x\in Z(\mathfrak m)$. 
Clearly $\beta(x)\supseteq\mathfrak m$, but $\mathfrak m$ is maximal, and hence
$\beta(x)=\mathfrak m$. 
Therefore $\beta$ is surjective as well. 
\end{proof}

\begin{notn} 
Let $\mathbf{BO}$ denote the category of Boolean rings where morphisms are ring homomorphisms, 
and let $\mathbf{CTS}$ denote the category of compact, totally separated topological spaces where morphisms are continuous maps. 
There are two contravariant functors
\[
\BB \colon \mathbf{CTS}\to\mathbf{BO},\quad X\mapsto \BB(X),
\]
which is well-defined as we saw in Example \ref{ex1.2}, and 
\[
\SB \colon \mathbf{BO}\to\mathbf{CTS},\quad R\mapsto\mathrm{Spec}(R),
\]
which is well-defined by Proposition \ref{spec}.
\end{notn}

\begin{thm}[Stone duality]\label{duality} 
The functors $\BB$ and $\SB$ are a pair of dualities of categories. 
\end{thm}
\begin{proof} 
By Theorem \ref{sigma}, the map $\sigma$ is a natural isomorphism between the identity of $\mathbf{BO}$ and $\BB \circ \SB$. 
By Theorem \ref{beta}, the map $\beta$ is a natural isomorphism between the identity of $\mathbf{CTS}$ and $\SB \circ \BB$.
\end{proof}

\begin{cor}\label{6.4/2016} 
Let $R$ be a Boolean ring. 
Then the following are equivalent:
\begin{enumerate}
\item[$(i)$] $R$ is finite.
\item[$(ii)$] $R$ is finitely generated as an $\mathbb F_2$-algebra.
\item[$(iii)$] $R$ is Noetherian.
\item[$(iv)$] $R$ is Artinian.
\item[$(v)$] $\mathrm{Spec}(R)$ is finite.
\item[$(vi)$] $R\cong\mathbb F_2^{X}$ for a finite set $X$.
\end{enumerate}
In this case $R\cong\mathbb F_2^{\mathrm{Spec}(R)}$.
\end{cor}
\begin{proof} 
Every Boolean ring is an $\mathbb F_2$-algebra, so if $R$ is finite, then it is finitely generated as an $\mathbb F_2$-algebra, and hence $(i)$ implies $(ii)$. Every finitely generated $\mathbb F_2$-algebra is Noetherian by Hilbert's basis theorem, so $(ii)$ implies $(iii)$. 
Every Boolean ring is zero dimensional by part $(ii)$ of Proposition \ref{basic}, so if it is Noetherian, it is Artinian by a standard theorem in commutative algebra (see \cite[Theorem 8.5 on page 90]{AM}). Therefore $(iii)$ implies $(iv)$. 
Every Artinian ring is a finite direct product of Artinian local rings (see \cite[Theorem 8.7 on page 90]{AM}), so its spectrum is finite. 
Therefore $(iv)$ implies $(v)$. 
Now let $R$ be a Boolean ring whose spectrum is finite. 
Since $\mathrm{Spec}(R)$ is totally separated by Proposition \ref{spec}, it is discrete. 
This implies that $\F_2^{\Spec(R)}$ is isomorphic to $\BB(\Spec(R)$ by the definition of the $\BB(X)$ made in Example \ref{ex1.2} 
since $\Spec(R)$ is discrete.  
By Theorem \ref{sigma}, $R$ is isomorphic to $\BB(\Spec(R))$ via $\sigma$. 
Thus, we get $R \cong \mathbb F_2^{\mathrm{Spec}(R)}$.   
In particular, $(v)$ implies $(vi)$. If $R\cong\mathbb F_2^{X}$ for a finite set $X$, then $R$ is clearly finite, so $(vi)$ implies $(i)$.
\end{proof}

\begin{notn} 
Let $\mathbf{FBO}$ denote the category of finite Boolean rings where morphisms are ring homomorphisms, and let $\mathbf{FTS}$ denote the category of finite, totally separated topological spaces where morphisms are continuous maps. Note that the latter is the same as the category of finite, discrete topological spaces. There are two restrictions of functors
\[
\BB|_{\mathbf{FTS}} \colon \mathbf{FTS}\to\mathbf{FBO},\quad X\mapsto \BB(X)
\]
and 
\[
\SB|_{\mathbf{FBO}} \colon \mathbf{FBO}\to\mathbf{FTS},\quad R \mapsto\mathrm{Spec}(R),
\]
where the latter is well-defined by Corollary \ref{6.4/2016}.
\end{notn}

\begin{thm}\label{fin-duality} 
The functors $\BB|_{\mathbf{FTS}}$ and
$\SB|_{\mathbf{FTS}}$ are a pair of dualities of categories. 
\end{thm}
\begin{proof} 
The restrictions of the natural isomorphisms $\beta$ and $\sigma$ onto $\mathbf{FTS}$ and $\mathbf{FBO}$ are the respective required natural isomorphisms. 
\end{proof}

\begin{defn} 
Let $R$ be a Boolean ring. 
We say that two elements $x,y\in R$ are {\it orthogonal} if $xy=0$. 
We say that an element $a\in R$ is an {\it atom} 
if it is non-zero and cannot be written as the sum of two non-zero orthogonal elements of $R$. 
The {\it support} of an element $x\in R$ is the subset $D(x) \subseteq \Spec(R)$ introduced in \eqref{eq:def_of_Df}.  
\end{defn}

\begin{lemma}\label{ortho} 
Let $R$ be a Boolean ring. Then the following holds:
\begin{enumerate}
\item[$(i)$] Two elements $x,y\in R$ are orthogonal if and only if the intersection of their support is empty. 
\item[$(ii)$] A non-zero element $a\in R$ is an atom if and only 
if its support cannot be written as the disjoint union of two non-empty open and closed subsets of $\Spec(R)$. 
\item[$(iii)$] Every two distinct atoms of $R$ are orthogonal to each other. 
\end{enumerate}
\end{lemma}
\begin{proof} 
For $\frp\in\Spec(R)$, we have $xy\notin\frp$ if and only if $x\notin\frp$ and $y\notin\frp$.
Hence $D(xy)=D(x)\cap D(y)$. 
Thus, the support of the product $xy$ is the intersection of the supports of $x$ and $y$. 
As explained in the first paragraph of the proof of Theorem \ref{sigma}, 
the nilradical of $R$ is zero and hence the intersection of all prime ideals of $R$ is zero by \cite[Proposition 1.8 on page 5] {AM}.  
Hence, for $a\in R$, $D(a)=\emptyset$ if and only if $a=0$. 
This shows that the support of an element of $R$ is empty if and only if the element is zero. 
Claim $(i)$ now follows. 
By part $(iii)$ of Proposition \ref{basic}, 
$(x,y)=(x+y)$, since $xy=0$.
Hence, as in the proof of Theorem \ref{sigma}, 
$D(a)=D(x)\cup D(y)$.
By part $(i)$, $D(x)\cap D(y)=\emptyset$. 
Therefore, $D(a)$ is the disjoint union $D(a)=D(x)\sqcup D(y)$.  
Also, as in the end of the proof of Proposition \ref{spec}, 
$\Spec(R)=D(x)\sqcup   D(1-x)$ and $\Spec(R)=D(y)\sqcup  D(1-y)$,
so $D(x)$ and $D(y)$ are open and closed subsets of $\Spec(R)$.  
%
Now assume that the support of $a$ is the disjoint union of the non-empty open and closed subsets $X$ and $Y$ of $\Spec(R)$. 
Let $f_X,f_Y \colon \Spec(R) \to \F_2$
be the characteristic functions of $X$ and $Y$, respectively, 
which are continuous by the assumption that $X$ and $Y$
are open and closed subsets of $\Spec(R)$. 
Then, as in the proof of Theorem \ref{sigma}, 
there exist $x,y\in R$ such that
$X=D(f_X)=D(x)$ and $Y=D(f_Y)=D(y)$. 
Since $D(\sigma(a))=D(a)=D(x)\sqcup D(y)=D(x+y)=D(\sigma(x+y))$, we have
$\sigma(a)=\sigma(x+y)$, so by Theorem \ref{sigma}, $a=x+y$.  
Let $a,b\in R$ be two atoms whose product is non-zero. 
Then $a=ab+a(1-b)$ and $aba(1-b)=a^2(b-b^2)=0$, so $a(1-b)=0$ by the definition of atoms. 
We get that $a=ab$. The same reasoning for $b$ shows that $b=ab$. 
Therefore $a=b$, and hence $(iii)$ holds. 
\end{proof}

\begin{cor}\label{atomic} 
Let $R$ be a finite Boolean ring. 
Then the atoms of $R$ form a natural basis of $R$ as an $\F_2$-vector space whose elements are orthogonal to each other. 
Moreover, every orthogonal basis consists of atoms. 
\end{cor}
\begin{proof} 
By Corollary \ref{6.4/2016},  the topological space
$\mathrm{Spec}(R)$ is finite, discrete, and $R\cong\mathbb F_2^{\mathrm{Spec}(R)}$. 
Therefore, an element of $R$ is an atom if and only if its support is a point by claim $(ii)$ of Lemma \ref{ortho}. 
These clearly form a basis of $R$ and they are orthogonal by claim $(iii)$ of Lemma \ref{ortho}. 

Now let $e_1,e_2,\ldots,e_n$ be an orthogonal basis. 
We need to show that each $e_i$ is an atom. 
We may assume without the loss of generality that $i=1$. 
Let $x\in R$ be such that $\sigma(x)$ is the characteristic function of an element of the support of $e_1$.  
Then $x$ is an atom. 
The supports of $e_1$ and $e_i$, $i\neq 1$, are disjoint since they are orthogonal by claim $(i)$ of Lemma \ref{ortho}. 
Hence $x$ and $e_i$ are orthogonal, so $xe_i=0$. 
Now write $x$ as a linear combination $\sum_{j\in J}e_j$ of the $e_i$ for $J \subseteq \{1,\ldots,n\}$.  
Since $0 = xe_i = \sum_{j\in J}e_je_i$ only if $i$ is not in $J$, 
the above argument implies $x=0$ or $x=e_1$. 
The former is not possible, since $x$ is not zero. 
Hence we get $x=e_1$, and $e_1$ is an atom by the first paragraph of the proof. 
\end{proof}


\section{All about profinite spaces}\label{sec:profinite_spaces}

In this section we collect the results we need about profinite topological spaces. 

\begin{defn} 
We say that a topological space is {\it profinite} if it is the projective limit of discrete, finite topological spaces. 
Recall that a topological space $X$ is {\it totally disconnected}  
if for every point $x \in X$ the connected component of $x$ in $X$ is $\{x\}$.
\end{defn}

Every totally separated topological space is totally disconnected, but the converse is not true: 
there are totally disconnected topological spaces which are not Hausdorff, while every totally separated topological space is Hausdorff. 
However, the following is true:

\begin{thm}\label{mama} 
Let $X$ be a topological space. 
Then the following are equivalent:
\begin{enumerate}
\item[(i)] $X$ is profinite.
\item[(ii)] $X$ is homeomorphic to a closed subspace of a product of discrete, finite topological spaces.
\item[(iii)] $X$ is compact, totally disconnected and Hausdorff.
\item[(iv)] $X$ is compact and totally separated.
\end{enumerate}
\end{thm}
\begin{proof} 
First assume that $X$ satisfies $(i)$. 
Recall that a category is called {\it small} if both its class of objects and class of morphisms
are sets and not proper classes. 
Let $\mathcal C$ be a small category of discrete, finite topological spaces whose projective limit is $X$. 
Let $\mathrm{Ob}(\mathcal C)$ denote the set of its objects, and let $\mathrm{Hom}_{\mathcal C}(A,B)$ denote the set of its morphisms for every $A,B\in\mathrm{Ob}(\mathcal C)$. 
By definition, $X$ is homeomorphic to the closed subspace
\begin{align*}
\left\{\prod_{C\in\mathrm{Ob}(\mathcal C)}x_C\in\prod_{C\in\mathrm{Ob}(\mathcal C)}C\mid f(x_A)=x_B\quad\left(\forall A,B\in\mathrm{Ob}(\mathcal C),\ \forall f\in\mathrm{Hom}_{\mathcal C}(A,B)\right)\right\}
\end{align*}
of $\prod_{C\in\mathrm{Ob}(\mathcal C)}C$, so $X$ satisfies $(ii)$.

Next assume that $X$ satisfies $(ii)$. 
Since finite topological spaces are compact, their direct product is also compact by Tychonoff's theorem. 
Since $X$ is a closed subspace of a compact space, it is compact, too. 
Moreover, the direct product of totally disconnected and Hausdorff topological spaces is totally disconnected and Hausdorff. 
Since discrete topological spaces are totally disconnected and Hausdorff, and every subspace of a totally disconnected and Hausdorff topological space is also totally disconnected and Hausdorff, we get that the same holds for $X$, too. 
Therefore $X$ satisfies $(iii)$.

Now we show that $(iii)$ implies $(iv)$. 
We will start with the following standard
\begin{lemma}\label{sep} 
Assume that $X$ is a Hausdorff compact topological space. 
Let $C,D\subset X$ be two disjoint closed subsets. 
Then there exist disjoint open sets $U\subset X$ containing $C$ and $V\subset X$ containing $D$. 
\end{lemma}
\begin{proof} 
First assume that $C$ consists of a single point $x\in X$. 
Because $X$ is Hausdorff for every $y\in D$, there are disjoint open sets $U_y\subset X$ containing $x$ and $V_y\subset X$ containing $y$. 
Because $D$ is closed, it is compact, so there is a finite subset $y_1,\ldots,y_n\in D$ such that $V_{y_1},\ldots,V_{y_n}$ cover $D$. 
Then $U=U_{y_1}\cap\cdots\cap U_{y_n}$ and $V=V_{y_1}\cup\cdots\cup V_{y_n}$ are disjoint open subsets containing $x$ and $D$, respectively. 

Now consider the general case. 
By the above for every $x\in C$ there are disjoint open sets $U_x\subset X$ containing $x$ and $V_x\subset X$ containing $D$. 
Because $C$ is closed, it is compact, so there is a finite subset $\{x_1,\ldots,x_n\} \subset C$ such that $U_{x_1},\ldots,U_{x_n}$ cover $C$. 
Then $U \coloneqq U_{x_1}\cup\cdots\cup U_{x_n}$ 
and $V \coloneqq V_{x_1}\cap\cdots\cap V_{x_n}$ are disjoint open subsets containing $C$ and $D$, respectively. 
%
\end{proof}

For every $x\in X$, let $Z(x)$ denote the set of all open and closed subsets of $X$ containing $x$, 
and let $Z_x$ denote the intersection of all elements of $Z(x)$. 
By the definition of a totally separated topological space, 
we need to show that if $y\in X$ is distinct from $x$ then $y\not\in Z_x$. 
Assume that this is not the case for some $y$. 
Then $Z_x$ contains at least two points, so it is not connected, since $X$ is totally disconnected. 
Therefore $Z_x$ is the disjoint union of two non-empty subsets $C,D\subset Z_x$ which are closed in $Z_x$.  
Since $Z_x$ is the intersection of closed subsets, it is closed in $X$. 
Therefore $C$ and $D$ are also closed in $X$. 
Hence by Lemma \ref{sep}, 
we can find an open set $U\subset X$ containing $C$ such that $\overline U \cap D = \emptyset$. 
Let $E$ be the complement of the union of $U$ and $V$ in $X$. 
It is closed in $X$, so it is compact since $X$ is compact. 
Since $U\cup V$ contains $Z_x$, the set $E$ is covered by the union of the complements of elements of $Z(x)$. 
Since $E$ is compact, it is already covered by the union of finitely many such sets. 
Therefore there are finitely many sets $Z_1,\ldots,Z_n \in Z(x)$ 
such that $Z=Z_1\cap\cdots\cap Z_n\subseteq U\cup V$. 

Note that since each $Z_i$ is both open and closed, the same also holds for the finite intersection $Z$. 
Since each $Z_i$ contains $x$, the set $Z$ also contains $x$. 
Both $Z\cap U$ and $Z\cap V$ are open in $Z$, 
and $Z$ is open in $X$, 
so $Z\cap U$ and $Z\cap V$ are open in $X$, too. 
Since $Z$ lies in the disjoint union of $U$ and $V$, 
every element of $X$ which is not in $Z\cap U$ is either not in $Z$ or it is in $Z\cap V$. 
Thus, we get $Z \cap U = X - [(X - Z) \cup (Z \cap V)]$. 
Since $Z$ is closed, $X - Z$ is open and we just showed $Z \cap V$ is open. 
Hence $Z\cap U$ is also closed since it is the complement of an open subset. 
The same argument with the roles of $U$ and $V$ reversed shows that also $Z\cap V$ is closed. 

One of $Z\cap U$ and $Z\cap V$ contain $x$, say $Z\cap U$. 
Then $Z\cap U\in Z(x)$ which means that $Z\cap U$ contains $Z_x$. 
But $Z\cap U\cap Z_x=Z\cap(U\cap Z_x)=Z\cap C=C$, which is a contradiction, 
since the complement of $C$ in $Z_x$ is $D$, 
and the latter is non-empty. 
So condition $(iii)$ implies condition $(iv)$.

Finally we show that $(iv)$ implies $(i)$. We start with the following
%
%
\begin{notn}
Let $X$ be a topological space. 
We denote by $\Rh(X)$ the set of open and closed equivalence relations $R$ on $X$ 
such that $X/R$ is finite and discrete. 
We may view $R \in \Rh(X)$ as an open and closed subset of $X \times X$ 
and we order elements in $\Rh(X)$ by inclusion, 
i.e., $R_1 \ge R_2$ in $\Rh(X)$ if and only if $R_1 \subseteq R_2$.  
This turns $\Rh(X)$ into a directed partially ordered set. 
To see that $\Rh(X)$ is directed, 
let $R_1, R_2 \in \Rh(X)$. 
Then $R_1 \cap R_2 \subset X \times X$ is open and closed. 
Moreover, the quotient maps $X/(R_1 \cap R_2) \to X/R_1$ and $X/(R_1 \cap R_2) \to X/R_2$ 
are continuous 
and induce a continuous map $X/(R_1 \cap R_2) \to X/R_1 \times X/R_2$. 
This map is injective 
since $(x,y) \in R_1$ and $(x,y) \in R_2$ if and only if $(x,y) \in R_1 \cap R_2$.  
Thus, we can consider $X/(R_1 \cap R_2)$ as a subspace of the finite and discrete space $X/R_1 \times X/R_2$,  
and hence $X/(R_1 \cap R_2)$  is finite and discrete, too. 
This shows $R_1 \le R_1 \cap R_2$ and $R_2 \le R_1 \cap R_2$ as required. 
\end{notn}
\begin{defn}\label{def:completion}
Let $X$ be a topological space. 
We define the {\it profinite completion} of $X$ to be the profinite space $\widehat X$ given by 
the projective limit $\widehat X \coloneqq \lim_{R\in \Rh(X)} X/R$ in the category of topological spaces. 
It is equipped with a continuous map $u_X \colon X \to \widehat X$ 
induced by the quotient maps $X \to X/R$ and the general properties of projective limits. 
\end{defn}

\begin{remark}\label{rem:completion_universal}
The continuous map $u_X \colon X \to \widehat X$ satisfies the following universal property:
For every continuous map $f \colon X\to Y$, where $Y$ is a discrete and finite topological space, 
there is a unique continuous map $\widehat f \colon \widehat X\to Y$ such that $f=\widehat f\circ u_X$.  
This follows from the fact that $f$ induces an equivalence relation $R_f \subset X \times X$ on $X$ 
given by $(x,x')\in R_f$ 
if and only if $f(x) = f(x')$. 
For $x\in X$, let $y = f(x)$. 
Since $Y$ is discrete, the preimage $f^{-1}(y)$ is both open and closed in $X$. 
Moreover, since $Y$ is finite, there are finitely many equivalence classes under $R_f$.  
Hence the quotient space $X/R_f$ is finite and discrete. 
This shows that $R_f \in \Rh(X)$. 
Let $\overline{f} \colon X/R_f \to Y$ be the induced map. 
The map $\widehat f$ is now defined as the composition of the canonical map $\widehat X \to X/R_f$ and 
the map $\overline{f}$. 
\end{remark}


Now we return to the proof of Theorem \ref{mama}. 
To show that $(iv)$ implies $(i)$ 
it will be sufficient to show that $u_X$ is a homeomorphism when $X$ is compact and totally separated. 
Since $X$ is compact and $\widehat X$ is Hausdorff by $(i) \Rightarrow (iii)$ 
of Theorem \ref{mama} which we have already shown, 
it will be sufficient to show that $u_X$ is a bijection by \cite[I \S 9.4, Corollary 2 on page 87]{BourbakiTop} 
which states that a continuous bijection from a compact space to a Hausdorff space is a homeomorphism. 
Let $x,y\in X$ be two distinct points. 
Since $X$ is totally separated, there is a continuous function 
$f \colon X \to \F_2$  
such that $f(x)\neq f(y)$ as we saw in the proof of Theorem \ref{beta}. 
Then $\widehat f(u_X(x))=f(x)\neq f(y)=\widehat f(u_X(y))$, so $u_X(x)\neq u_X(y)$. 
Therefore $u_X$ is injective. 
In order to see that $u_X$ is also surjective, we will need the following

\begin{lemma}\label{sep2} 
Let $X$ be a compact and totally separated space, and let $C_1,\ldots,C_n\subset X$ be pairwise disjoint closed subsets. 
Then there exist pairwise disjoint open and closed subsets $U_1,U_2,\ldots,U_n\subset X$ such that $U_i$ contains $C_i$ for each $i$, and $\cup_iU_i=X$. 
\end{lemma}
\begin{proof} 
We first show that it is enough to prove the claim when $n=2$ as the general case follows by induction. 
Indeed let $n>2$ be such that we already know the claim for every $m<n$. 
Apply the case $m=2$ to the closed subsets $D_1=C_1$ and $D_2=C_2\cup\cdots\cup C_n$ 
to get two disjoint open and closed subsets $W_1,W_2\subset X$ such that $D_1\subseteq W_1$, $D_2\subseteq W_2$, and $W_1\cup W_2=X$. 
Then apply the case $m=n-1$ to  $C_2,\ldots,C_n$ inside $W_2$, 
which is possible since $W_2$ is closed in $X$, so it is compact and totally separated. 
Therefore there exist pairwise disjoint open and closed subsets $V_2,\ldots,V_n\subset W_2$ 
such that $V_i$ contains $C_i$ for each $i\geq2$, and $\cup_iV_i=W_2$. 
Since each $V_i$ is open and closed in an open and closed subset of $X$, it is also open and closed in $X$. 
Therefore $U_1=W_1$ and $U_i=V_i$ for $i\geq2$ have the required properties. 

Now consider the case $n=2$. 
Since $X$ is totally separated by assumption, it is also Hausdorff. 
By Lemma \ref{sep}, 
there then exist disjoint open sets $V_1\subset X$ containing $C_1$ and $V_2\subset X$ containing $C_2$. 
Because $X$ is totally separated, its open and closed sets form a subbasis for its topology, 
so for every $x\in C_1$ there is a closed and open subset $V_x\subset V_1$ containing $x$. 
Because $C_1$ is closed, it is compact, so there is a finite subset $x_1,\ldots,x_n\in C$ such that $V_{x_1},\ldots,V_{x_n}$ cover $C$. 
Their union $U_1=V_{x_1}\cup\cdots\cup V_{x_n}$ is the union of open and closed subsets, so it is also open and closed, 
and it is contained in $V_1$.  
Its complement $U_2$ is also open and closed and contains $C_2$, so $U_1$ and $U_2$ have the required properties. 
\end{proof}

Now assume that $u_X$ is not surjective, so there is an $x\in\widehat X$ such that $x\not\in u_X(X)$. 
Since $X$ is compact and $\widehat X$ is Hausdorff by $(i) \Rightarrow (iii)$ 
of Theorem \ref{mama} which we have already proven,  
the image $u_X(X)$ is closed in $\widehat X$ by \cite[I \S 9.4, Corollary 2 on page 87]{BourbakiTop} 
which states that a continuous map from a compact space to a Hausdorff space is closed. 
So by Lemma \ref{sep2} and since $\widehat X$ is compact and totally separated by $(i) \Rightarrow (iv)$ of Theorem \ref{mama}, 
which we already have proved, 
there are disjoint open and closed subsets $U_1,U_2\subset\widehat X$ such that $u_X(X)\subseteq U_1$ and $x\in U_2$. Let $f:\widehat X\to\mathbb F_2$ the characteristic function of $U_1$. 
It is continuous, since $U_1$ is open and closed. 
Let $g \colon \widehat X\to\mathbb F_2$ be the identically $1$ function. 
It is continuous, and $f\circ u_X=g\circ u_X$. 
However $f\neq  g$, as $f(x)\neq g(x)$, which violates the universal property. 
This is a contradiction, so $(iv)$ implies $(i)$.
\end{proof}


\begin{remark}\label{rem:completion_comment}
We note that the profinite completion of $X$ as a topological space in Definition \ref{def:completion} 
differs from the profinite completion of $X$ as a set. 
The latter is the completion of $X$ as a space with the discrete topology. 
For example, if $X$ is connected then the only non-empty open and closed equivalence relation in $\Rh(X)$ is $X \times X$. 
As a set, however, there may be many equivalence relations on $X$ such that $X/R$ is finite. 
Moreover, if $X$ is a profinite space, the proof of $(iv) \Rightarrow (i)$ of Theorem \ref{mama} 
shows that $u_X$ is a homeomorphism. 
This may not be the case if we first equip $X$ with the discrete topology. 
%
\end{remark}


Next, we show that profinite completion and the composition of functors $\Spec \circ \BB$ of Section \ref{sec:Boolean_rings} 
yield homeomorphic spaces. 

\begin{notn} 
Let $X$ be again an arbitrary topological space, and let
$\beta \colon X\to\mathrm{Spec}(\BB(X))$ be the map introduced in Notation \ref{def1.7}. 
By Proposition \ref{spec} the space $\mathrm{Spec}(\BB(X))$ is compact and totally separated, 
so it is profinite by $(iv) \Rightarrow (i)$ of Theorem \ref{mama}. 
So it is a projective limit of finite, discrete spaces, and hence by the universal property of the profinite completion 
there is a unique continuous map $\widehat{\beta} \colon \widehat X\to\mathrm{Spec}(\BB(X))$ 
such that $\beta=\widehat{\beta}\circ u_X$. 
\end{notn}

\begin{thm}\label{thm:comparing_completion_and_Spec_BB}
The map $\widehat{\beta} \colon \widehat X\to\mathrm{Spec}(\BB(X))$ is a homeomorphism. 
\end{thm}
\begin{proof} 
According to the universal property of the profinite completion it will be sufficient to show that for every continuous map $f \colon X\to Y$, where $Y$ is a discrete, finite topological space, 
there is a unique continuous map $\widetilde f \colon \mathrm{Spec}(\BB(X))\to Y$ such that $f=\widetilde f\circ\beta$. 
For every continuous map $m \colon T \to Q$ of topological spaces, 
let $m^* \colon \BB(Q) \to \BB(T)$ denote the induced ring homomorphism. 
Since $\beta^* \colon \BB(\Spec(\BB(X)))\to \BB(X)$ is the inverse of $\sigma$ in Notation \ref{notn:sigma_map} for $R= \BB(X)$, 
$\beta^*$ is an isomorphism by Theorem \ref{sigma}. 
Hence there is a unique ring 
homomorphism $r \colon \BB(Y) \to \BB(\Spec(\BB(X)))$ such that $\beta^*\circ r=f^*$. 
By Theorem \ref{duality} there is a unique continuous map $\widetilde f \colon \mathrm{Spec}(\BB(X)) \to Y$ such that $r=\widetilde f^*$. 
Then $f^*=\beta^* \circ \widetilde f^*=(\widetilde f\circ\beta)^*$, so it will be sufficient to show that every continuous map $h \colon X \to Y$ is uniquely determined by $h^*$. 
Since for every $x\in X$ the point $h(x)$ is uniquely determined by the ideal
$$\{a \in \BB(Y)\mid a(h(x))=0\}=
\{a \in \BB(Y)\mid h^*(a)(x)=0\},$$
this claim is clear. 
\end{proof}


\section{Profinite principal $G$-bundles}\label{sec:prof_G_bundles}

The purpose of this section is to prove Theorem \ref{section} on the existence of sections for profinite principal $G$-bundles. 

\begin{defn} 
Let $A$ be a topological space and let $m \colon A\to B$ be a map. 
The {\it quotient topology} on $B$ with respect to $m$ is defined in the following way: 
a subset $U\subseteq B$ is open if and only if $m^{-1}(U)
\subseteq A$ is open. 
Now let $A$ be a topological space equipped with a left action of a group $H$. 
Let $H\backslash A$ denote the quotient of $A$ with respect to the left action of $H$ equipped with the quotient topology. 
\end{defn}

\begin{lemma}\label{quotienting} 
The following hold:
\begin{enumerate}
\item[$(a)$] Let $A$ be a topological space and let $m \colon A\to B$ be a map. 
Let $B$ be equipped with the quotient topology with respect to $m$. 
Then a map $h \colon B\to C$ of topological spaces is continuous if and only if the composition $h\circ m \colon A \to C$ is continuous.  
\item[$(b)$] Let $A$ be a topological space equipped with a left action of a group $H$. Then the quotient map $A\to H\backslash A$ is open, that is, it maps open sets to open sets. 
\item[$(c)$] Let $A,B$ be two topological spaces both equipped with a left action of a group $H$. Then the product topology on $H\backslash A\times
H\backslash B$ is the quotient topology with respect to the quotient map
$$A\times B\longrightarrow (H\times H)\backslash(A\times B)=
H\backslash A\times H\backslash B.$$
\end{enumerate}
\end{lemma}
\begin{proof} 
We first prove $(a)$. 
Since the composition of continuous functions is continuous, the map $h\circ m$ is continuous if $h$ and $m$ are. 
However, $m$ is continuous by construction, so $h\circ m$ is continuous if $h$ is. 
On the other hand, if $h\circ m$ is continuous and $U\subseteq C$ is open, then $(h\circ m)^{-1}(U)\subseteq A$ is open. 
Therefore $m^{-1}(h^{-1}(U))=(h\circ m)^{-1}(U)$ is also open, so by definition $h^{-1}(U)\subseteq B$ is open. 
Hence $h$ is continuous, and now claim $(a)$ is clear.

Next we show $(b)$. Let $U\subseteq A$ be open, and let $q \colon A\to H\backslash A$ denote the quotient map. 
Then $q^{-1}(q(U))=\bigcup_{\gamma\in H}\gamma U$. 
Since each $\gamma\in H$ acts as a homeomorphism on $A$ we get that $\gamma U$ is open. 
Therefore their union $q^{-1}(q(U))$ is also open, and hence $q(U)$ is open by the definition of the quotient topology. 
Claim $(b)$ is now clear. 

Finally, we show that $(c)$ holds. 
Let $q_A \colon A\to H\backslash A$ and 
$q_B \colon B\to H\backslash B$ be the respective quotient maps. 
Since the map
\[
q_A\times q_B \colon A\times B\longrightarrow H\backslash A\times H\backslash B
\]
is continuous with respect to the product topologies, the preimage of any open subset is open. 
Therefore we only need to show that if $(q_A\times q_B)^{-1}(U)\subseteq A\times B$ for a subset $U\subseteq H\backslash A\times H\backslash B$ is open, 
then $U$ is open in the product topology. 
In this case, $(q_A\times q_B)^{-1}(U)$ is the union of sets of the form $V\times W$, where $V\subseteq A$ and $W\subseteq B$ are open. 
By part $(b)$ the set $(q_A\times q_B)(V\times W)=q_A(V)\times q_B(W)$ is open in the product topology. 
Since $U$ is the union of such subsets, it is open, and $(c)$ follows. 
\end{proof}

\begin{lemma}\label{booletrick} 
Let $X$ be a profinite space and let $\mathcal U$ be an open covering of $X$. Then there is a finite open covering $\mathcal V$ of $X$ consisting of pairwise disjoint open and closed subsets which is subordinate to $\mathcal U$, 
i.e., each $V \in \mathcal V$ is contained in some $U \in \mathcal U$. 
\end{lemma}
\begin{proof} 
By $(i) \Rightarrow (iv)$ of Theorem \ref{mama} and Lemma \ref{sep2}, 
the open and closed subsets of $X$ form a subbasis of the topology on $X$. 
Therefore there is an open covering $\mathcal W$ of $X$ consisting of open and closed subsets which is subordinate to
$\mathcal U$. 
Since $X$ is compact, we may assume without the loss of generality that $\mathcal W$ is finite. 
Let $R$ be the subring of $\BB(X)$ generated by the characteristic functions of the elements of
$\mathcal W$. 
Since $R$ is finitely generated, it is finite by Corollary  \ref{6.4/2016}.

Let $\mathcal V$ be the collection of open and closed subsets of $X$
whose characteristic functions are the atoms of $R$.
By Theorem \ref{beta} , 
we can identify them as the supports of the atoms of $R$. 
We claim that $\mathcal V$ satisfies the required properties.  
Because $R$ is finite, the set $\mathcal V$ is also finite. 
Since distinct atoms of $R$ are orthogonal by part $(iii)$ of Lemma \ref{ortho}, 
their support is pair-wise disjoint by part $(i)$ of Lemma \ref{ortho}. 
Let $e_1,e_2,\ldots,e_n$ be the atoms of $R$. 
By Corollary \ref{atomic}, we then have $1=e_1+\ldots+e_n$, so the union of their supports is $X$. 
Also the elements of $\mathcal V$ are open, since they are the supports of elements of $\BB(X)$.

Now let $e$ be an atom of $R$ and pick an $x$ in its support $V$. 
Since $\mathcal W$ is a covering, 
there is a $U\in\mathcal W$ such that $x\in U$. 
The characteristic function $f$ of $U$ is in $R$, so it is the sum of distinct atoms of $R$. 
Therefore, $ef$ is either $0$ or $e$ by the orthogonality of atoms. 
In the first case, the intersection of the supports $U$ and $V$ is empty, but both contain $x$, a contradiction. 
Hence $ef=e$, so $U$ contains $V$. 
Hence $\mathcal V$ is subordinate to $\mathcal W$ and since $\mathcal W$ is subordinate to $\mathcal U$, 
 $\mathcal V$ is subordinate to $\mathcal U$. 
\end{proof}

Recall that a {\it section} of a continuous map of topological spaces $f \colon Y\to X$ 
is a continuous map $s \colon X\to Y$ such that $f\circ s$ is the identity map of $X$. 

\begin{prop}\label{subordintate} 
Let $f \colon Y\to X$ be a continuous map such that $X$ is profinite and every $x\in X$ has an open neighbourhood $U$ such that $f|_{f^{-1}(U)} \colon f^{-1}(U)\to U$ has a section. 
Then $f \colon Y\to X$ has a section. 
\end{prop}
\begin{proof} 
By assumption, there is an open cover $\mathcal U$ of $X$ such that
$f|_{f^{-1}(U)}\colon f^{-1}(U)\to U$ has a section for each $U \in \mathcal U$. 
By Lemma \ref{booletrick} there is a finite open covering $\mathcal V$ of $X$ consisting of pairwise disjoint open and closed subsets which is subordinate to $\mathcal U$. 
Let $V \in \mathcal V$ and pick a $U\in \mathcal U$ which contains $V$. By assumption there is a section of $f|_{f^{-1}(U)} \colon f^{-1}(U)\to U$; its restriction  to $V$ is a section $s_V \colon V\to f^{-1}(V)$ of $f|_{f^{-1}(V)} \colon f^{-1}(V)\to V$. 
Since the elements of $\mathcal V$ form an open and pairwise disjoint covering of $X$, 
the union $\bigcup_{V\in\mathcal V}s_V$ of these sections is a section of $f \colon Y\to X$. 
\end{proof}

\begin{prop}\label{free} 
Let $G$ be a compact group, and let $X$ be a profinite space equipped with a continuous group action $g \colon G\times X\to X$. 
Assume that the action is free, 
i.e., no non-trivial element of $G$ fixes a point of $X$. 
Then both $G$ and $G\backslash X$ are profinite. 
\end{prop}
\begin{proof} 
For the sake of simple notation, 
in the sequel we will denote any left action of any group on any set by multiplication on the left, 
if this does not lead to confusion. 
Since the action $g$ is continuous and free, for every $x\in X$ the map $\gamma\mapsto\gamma x$ from $G$ onto the $G$-orbit of $x$ is continuous and injective. 
As $G$ is compact and $X$ is Hausdorff, 
we get that the image of this map is closed by \cite[I \S 9.4, Corollary 2 on page 87]{BourbakiTop}. 
Moreover, the map itself is a homeomorphism from $G$ onto its image. 
As $G$ is homeomorphic to a closed subspace of a profinite space, 
it is also profinite by the equivalence $(i) \iff (ii)$ of Theorem \ref{mama}. 

Let $q \colon X\to G\backslash X$ denote the quotient map. 
Since $q$ is surjective, continuous, and $X$ is compact, we get that $G\backslash X$ is compact by \cite[I \S 9.4, Theorem 2 on page 87]{BourbakiTop}. 
Let $x,y\in G\backslash X$ be two arbitrary distinct points. 
The preimages $C=q^{-1}(x)$ and $D=q^{-1}(y)$ are disjoint $G$-orbits, and they are also closed by the previous paragraph. 
Therefore, by $(i) \Rightarrow (iv)$ of Theorem \ref{mama} and Lemma \ref{sep2}, there exist disjoint open and closed subsets $U,V\subset X$ such that $U$ contains $C$ and $V$ contains $D$. 

Set $U' \coloneqq \bigcup_{\gamma\in G}\gamma U$. 
Since each $\gamma\in G$ acts as a homeomorphism, we get that $U'$ is the union of open subsets, so it is open. 
It is also the continuous image of $G\times U$. 
Both $G,U$ are compact since the latter is closed in a profinite space. 
Hence their product $G \times U$ is also compact by Tychonoff's theorem. 
Thus, since $X$ is Hausdorff, $U'$ is closed by \cite[I \S 9.4, Corollary 2 on page 87]{BourbakiTop}. 
Also $U'\cap D$ is empty; if it were not then $\gamma U\cap D\neq\emptyset$ for some $\gamma\in G$, 
but then $U\cap D=\gamma^{-1}(\gamma U\cap D)\neq \emptyset$ using that $D$ is $G$-invariant. 
This is a contradiction. 
A similar argument shows that $V'=\bigcup_{\gamma\in G}\gamma V$ is open and closed, and $V'\cap C$ is empty. 

Therefore $W=U'\cap V'$ is open and closed, and $W$ is disjoint from $C\cup D$. 
Therefore its complement $U''=U'-W$ in $U'$ is also open and closed, and contains $C$ since $U'$ does. 
Similarly $V''=V'-W$ is also open and closed, and contains $D$. 
Moreover, both $U'$ and $V'$ are $G$-invariant, so the same holds for $W$, and hence for $U''$ and $V''$. 
Because these sets are disjoint, their images $q(U'')$ and $q(V'')$ are also disjoint. 
They are also open by part $(b)$ of Lemma \ref{quotienting}, and since $x,y$ are in $q(U'')$ and $q(V'')$, respectively, and were arbitrary, we get that $G\backslash X$ is at least Hausdorff. 

Going back to the situation above, since $U''$ and $V''$ are closed in a profinite space, they are compact. 
Hence their images $q(U'')$ and $q(V'')$ are closed by \cite[I \S 9.4, Corollary 2 on page 87]{BourbakiTop}  
since $G\backslash X$ is Hausdorff. 
Therefore they are a pair of disjoint open and closed neighbourhoods of $x$ and $y$. 
Since $x$ and $y$ were arbitrary, we get that $G\backslash X$ is even totally separated, so it is profinite by $(iv) \Rightarrow (i)$ of Theorem \ref{mama}. 
\end{proof}

\begin{defn}\label{bundle-def} 
Let $X$ be a topological space and let $G$ be a compact group. 
A {\it profinite principal $G$-bundle over $X$} is a continuous map $f \colon Y\to X$ equipped with a free continuous group action $g \colon G\times Y\to Y$ 
such that $Y$ is profinite, the map $f$ is surjective and the preimage of each $x\in X$ is a $G$-orbit. 
\end{defn}

\begin{thm}\label{section} 
Let $X$ be a Hausdorff space and let $G$ be a compact group. 
Then every profinite principal $G$-bundle over $X$ has a section. 
\end{thm}

\begin{rems}\label{rem:4.8v1}
$(i)$ In the situation of Theorem \ref{section},  
 both $G$ and $X$ are profinite. 
This is immediate for $G$ from Proposition \ref{free}. 
Moreover there is a unique bijection $\iota \colon G\backslash Y\to X$ such that $f=\iota\circ q$, 
where $q \colon Y\to G\backslash Y$ denotes the quotient map. 
Since $f$ is continuous, we get that $\iota$ is a continuous bijection by part $(a)$ of Lemma \ref{quotienting}. 
Since $G\backslash Y$ is compact by Proposition \ref{free} and $X$ is Hausdorff, we get that $\iota$ is a homeomorphism by \cite[I \S 9.4, Corollary 2 on page 87]{BourbakiTop}. 
Moreover, $G\backslash Y$ is actually profinite by Proposition \ref{free}, so the same holds for $X$, too. 

$(ii)$ The reader might wonder if it is true that each continuous, surjective map $f \colon X \to Y$ of profinite spaces has a section. 
It turns out that those $Y$ which have this property for every such $f$ have a name: 
{\it extremally disconnected spaces}. 
They can be characterised by the following property: 
every set of open and closed subsets of $Y$ has a supremum with respect to inclusion (see the Folk Theorem of \cite{Gl} on page 485). 
Some spaces, such as the \v Cech--Stone compactification of discrete spaces have this property, but there are many profinite spaces which do not.  

$(iii)$ Theorem \ref{section} was already stated by Morel, see the remark after Lemma 4 of \cite{Mo} on page 359. 
However, no proof was given, just a remark that the strategy of the proof of Proposition 1 of \cite{Se} on page 4 works. 
This is what we will do, but for the convenience of the reader we will give a detailed argument. 
\end{rems}

\begin{proof}[Proof of Theorem \ref{section}] We start with the proof of the following significant special case:

\begin{prop}\label{finite-g} 
The theorem holds when $G$ is finite.
\end{prop}
\begin{proof} 
By Proposition \ref{subordintate} it will be sufficient to show that every $x\in X$ has an open neighbourhood $U$ such that $f|_{f^{-1}(U)} \colon f^{-1}(U)\to U$ has a section. 
Since $f^{-1}(x)$ is homeomorphic to $G$, it is non-empty, so there is a $y\in Y$ such that $f(y)=x$. 
By \cite[I \S 8.2, Proposition 4 on page 77]{BourbakiTop}, every finite subset of a Hausdorff space is closed. 
Because $G$ is finite and $Y$ is Hausdorff by $(i) \Rightarrow (iii)$ of Theorem \ref{mama}, 
this implies that every point of $Y$ is closed.  
Thus, by by $(i) \Rightarrow (iv)$ of Theorem \ref{mama} and Lemma \ref{sep2}, 
there exist pairwise disjoint open and closed subsets $U_{\gamma}\subset Y$ for all $\gamma\in G$ such that $U_{\gamma}$ contains $\gamma y$ for each $\gamma$. 

Set $V \coloneqq  \bigcap_{\gamma\in G}\gamma^{-1}U_{\gamma}$. 
It is the intersection of finitely many open and closed subsets, so it is also open and closed. 
For every $\gamma\in G$ we have $\gamma V\subseteq\gamma(\gamma^{-1}U_{\gamma})=U_{\gamma}$, so $V\cap\gamma V=\emptyset$ for every $\gamma\neq1$ in $G$. 
Therefore the restriction of $f$ to $V$ is injective. 
Since $V$ is closed, it is compact. 
Hence, since $X$ is Hausdorff by $(i) \Rightarrow (iii)$ of Theorem \ref{mama}, $f(V)$ is closed in $X$ by \cite[I \S 9.4, Corollary 2 on page 87]{BourbakiTop}. 
As a closed subspace of the Hausdorff space $X$, the image $f(V)$ is Hausdorff as well. 
Since $V$ is compact, $f(V)$ is Hausdorff and $f|_V$ is a continuous bijection, 
$f|_V$ is a homeomorphism onto its image by \cite[I \S 9.4, Corollary 2 on page 87]{BourbakiTop}. 
Thus, $f|_V$ has a continuous inverse $f(V)\to V$. 
Since $y=\gamma^{-1}(\gamma y)\in\gamma^{-1}U_{\gamma}$, we get that $y\in\gamma^{-1}U_{\gamma}$ for every $\gamma\in G$, so $y\in V$, and hence $x\in f(V)$.

Therefore it will be sufficient to show that $f(V)$ is open in $X$. 
Set $Z \coloneqq \bigcup_{\gamma\in G}\gamma V$. 
It is clearly $G$-invariant, and since it is the union of open subsets, it is also open. 
Therefore its complement $W=Y-Z$ in $Y$ is $G$-invariant and closed. 
Moreover, $f(V)=f(Z)$, as $f$ is $G$-equivariant, i.e., $f(\gamma y) = f(y)$ for each $\gamma \in G$ and $y \in Y$, 
and $f$ is surjective, so the complement of $f(V)$ in $X$ is $f(W)$. 
Since $W$ is closed, it is compact, and $X$ is Hausdorff, so $f(W)$ is closed by \cite[I \S 9.4, Corollary 2 on page 87]{BourbakiTop}. 
Therefore its complement $f(V)$ is open. 
\end{proof}


\begin{notn} 
Let $N$ be a closed normal subgroup of $G$ and let $G_N=N\backslash G$ denote the quotient: it is a profinite group. 
Let $r_N \colon G\to G_N$ be the quotient map, which is a continuous group homomorphism. 
There is a unique map $f_N \colon N\backslash Y\to X$ such that $f \colon Y\to X$ is the composition of the quotient map $q_N \colon Y\to N\backslash Y$ and $f_N$. 
There is a unique group action $G\times N\backslash Y\to N\backslash Y$ which makes $f_N$ a $G$-equivariant map. 
The restriction of this action onto $N$ is trivial, so it induces a group action $g_N \colon G_N\times N\backslash Y\to N\backslash Y$.
\end{notn} 


\begin{prop}\label{prop:4.11v1}
The map $f_N \colon N\backslash Y\to X$ equipped with the group action $g_N \colon G_N\times N\backslash Y\to N\backslash Y$ is a profinite principal $G_N$-bundle over $X$. 
\end{prop}
\begin{proof}  Since the composition $f=f_N\circ q_N$ is continuous, we get that the map $f_N$ is continuous by part $(a)$ of Lemma \ref{quotienting}. 
By part $(c)$ of Lemma \ref{quotienting} the topology on $G_N\times N\backslash Y$ is the quotient topology with respect to $r_N\times q_N$, so $g_N$ is continuous if $g_N\circ(r_N\times q_N)$ is continuous 
by part $(a)$ of Lemma \ref{quotienting}. 
However, the composition $q_N\circ g$ is continuous, so by the commutativity of the following diagram:
$$\xymatrix{G\times Y\ar[r]^g \ar[d]_{r_N\times q_N}  & Y\ar[d]^{q_N}  \\
G_N\times N\backslash Y \ar[r]^{\ g_N} & N\backslash Y}$$
the group action $g_N$ is continuous. 
Clearly the action $g_N$ is free, the map $f_N$ is surjective, and the preimage of each $x\in X$ is a $G_N$-orbit. 
Finally, $N\backslash Y$ is profinite by Proposition \ref{free} since the action of $N$ on $Y$ is free and $N$ is compact.
\end{proof}


\begin{notn}\label{notn:4.12v1}
Let $M,N$ be a pair of closed normal subgroups of $G$ such that $M\subseteq N$. 
There is a unique map $f_{M,N} \colon M\backslash Y\to N\backslash Y$ such that $q_N \colon Y\to N\backslash Y$ is the composition of the quotient map $q_M \colon Y\to M\backslash Y$ and $f_{M,N}$. Let $\mathcal S$ denote the set whose elements are ordered pairs
$(N,s)$, where $N$ is a closed normal subgroup of $G$ and $s$ is a section of $f_N \colon N\backslash Y\to X$. 
Let $\geq$ denote the binary relation on
$\mathcal S$ such that $(M,r)\geq(N,s)$ if and only if $M\subseteq N$, and $f_{M,N}\circ r=s$. 
Since for every triple $L\subseteq M\subseteq N$ of closed normal subgroups of $G$ we have $f_{M,N}\circ f_{L,M}=f_{L,N}$, we get that $\geq$ is a partial ordering on $\mathcal S$.
\end{notn}


\begin{prop} 
The partially ordered set $\mathcal S$ has a maximal element. 
\end{prop}
\begin{proof} 
For $N=G$, the map $f_N$ is a bijection from a compact space onto a Hausdorff topological space, 
so it is a homeomorphism by part $(i)$ of Remark \ref{rem:4.8v1}.  
Therefore its inverse is a section, and hence $\mathcal S$ is not empty. 
So by Zorn's lemma we only need to show that every chain $\mathcal C\subseteq\mathcal S$ has a maximal element. 
Set $C \coloneqq \bigcap_{(N,s)\in\mathcal C}N$. 
Since $C$ is the intersection of closed normal subgroups, it is also a closed normal subgroup of $G$. 

For every $(N,s)\in\mathcal C$, let $\Gamma(N,s)\subseteq C\backslash Y$ denote $f_{C,N}^{-1}(s(X))$, 
the preimage of the section $s \colon X\to N\backslash Y$ with respect to $f_{C,N} \colon C\backslash Y\to N\backslash Y$. 
By Proposition \ref{prop:4.11v1}, $N\backslash Y$ is profinite and hence $N\backslash Y$ is Hausdorff by $(i) \Rightarrow (iii)$ of Theorem \ref{mama}.   
Since $X$ is compact, the image $s(X)$ is closed by \cite[I \S 9.4, Corollary 2 on page 87]{BourbakiTop}.  
As $f_{C,N}$ is continuous, the preimage $\Gamma(N,s)$ is also closed. 
Therefore, their intersection $\Gamma\subseteq C\backslash Y$ is closed. 
By Proposition \ref{prop:4.11v1}, $C\backslash Y$ is profinite and hence it is compact by $(i) \Rightarrow (iii)$ of Theorem \ref{mama}.   
Hence the closed subset $\Gamma$ is also compact. 
Hence it will be sufficient to show that the restriction $f_C|_{\Gamma} \colon \Gamma\to X$, 
which is continuous, is also bijective. 
For, by \cite[I \S 9.4, Corollary 2 on page 87]{BourbakiTop} and since $X$ is Hausdorff and $\Gamma$ is compact, 
${f_C}|_\Gamma$ is a homeomorphism if it is bijective,
so in this case its inverse is a section $X\to\Gamma$, 
which combined with the inclusion $\Gamma \subseteq C\backslash Y$
gives the desired section $X\to C\backslash Y$.

Fix an $x\in X$. 
 Then, for every $(N,s)\in\mathcal C$, the intersection
$\Gamma(N,s)\cap f_C^{-1}(x) = f_{C,N}^{-1}(s(X))\cap f_C^{-1}(x)$ is non-empty
since $f_C=f_N\circ f_{C,N}$ and $s$ is a section of $f_N$. 
Moreover, $f_C^{-1}(x)$ is a finite subset of $C\backslash Y$ 
which is a Hausdorff space by Proposition \ref{prop:4.11v1} and $(i) \Rightarrow (iii)$ of Theorem \ref{mama} 
and hence $f_C^{-1}(x)$ is closed 
by \cite[I \S 8.2, Proposition 4 on page 77]{BourbakiTop}. 
Thus, for every $(N,s)\in\mathcal C$, the intersection 
$\Gamma(N,s)\cap f_C^{-1}(x)$  is a non-empty closed subset, 
and these form a descending chain with respect to inclusion. 
Hence, by the compactness of $C\backslash Y$, their intersection $\Gamma\cap f_C^{-1}(x)$ is non-empty. 
Therefore $f_C|_{\Gamma} \colon \Gamma\to X$ is surjective. 

Now assume that we have two distinct elements $y,z\in\Gamma\cap f_C^{-1}(x)$.  
Then there is a unique $1 \neq \gamma \in C \backslash G$ such that $z = \gamma y$. 
Since $C$ is the intersection $\bigcap_{(N,s)\in\mathcal C}N$, there is an $(N,s)\in\mathcal C$ such that the image of $\gamma$ under the quotient homomorphism $C\backslash G\to N\backslash G$ is not $1$. 
By Proposition \ref{prop:4.11v1}, both $f_C \colon C \backslash Y \to X$ and $f_N \colon N \backslash Y \to X$ 
are profinite principal bundles with respect to $C \backslash G$ and $N \backslash G$, respectively, 
and the diagram 
\begin{align*}
\xymatrix{
C \backslash Y \ar[rr]^-{f_{C,N}} \ar[dr]_-{f_C} & & N \backslash Y \ar[dl]^-{f_N} \\
& X &
}
\end{align*}
commutes. 
Moreover, the map $f_{C,N}$ is $C \backslash G$-equivariant, 
i.e., $f_{C,N}(\widetilde{\gamma}\widetilde{y}) = \widetilde{\gamma}f_{C,N}(\widetilde{y})$ for all $\widetilde{\gamma} \in C \backslash G$ and $\widetilde{y} \in N \backslash Y$. 
%
Since the action of $N\backslash G$ on $N \backslash Y$ is free by Proposition \ref{prop:4.11v1}, 
this shows that $f_{C,N}(y)$ and $f_{C,N}(\gamma y) = \gamma f_{C,N}(y)$ are still distinct.  
However, since $y$ and $z=\gamma y$ lie in $\Gamma$, which is the intersection of all $f^{-1}_{C,N}(s(X))$, 
and by construction of the partial ordering on $\mathcal S$ in Notation \ref{notn:4.12v1}, 
both $f_{C,N}(y)$ and $f_{C,N}(\gamma y)$ lie in the intersection of $s(X)$ and $f_N^{-1}(x)$, which consists of the single point $s(x)$; 
a contradiction. 
Therefore $f_C|_{\Gamma} \colon \Gamma\to X$ is injective, too.  
\end{proof}

Now let $(N,s)$ be a maximal element of $\mathcal S$. 
If we have $N=\{1\}$, then the theorem holds. 
So let us assume that this is not the case, and pick a non-zero $\gamma\in N$. 
Since $G$ is Hausdorff by $(i) \Rightarrow (iii)$ of Theorem \ref{mama} 
and since the open normal subgroups form a basis of the topology on $G$ by \cite[Theorem 2.1.3]{RZ}, 
there is an open normal subgroup $P\subset G$ such that $\gamma\not\in P$. 
Let $M=N\cap P$. 
Since $M$ is the intersection of closed normal subgroups, as every open subgroup is closed in a profinite group, 
$M$ is a closed normal subgroup as well, and
$\Gamma=M\backslash N$ is finite. 
Let $\Delta\subseteq M\backslash Y$ denote $f_{M,N}^{-1}(s(X))$, the preimage of the section $s \colon X\to N\backslash Y$ with respect to $f_{M,N} \colon M\backslash Y\to N\backslash Y$. 
By Proposition \ref{prop:4.11v1}, both $f_M \colon M \backslash Y \to X$ and $f_N \colon N \backslash Y \to X$ 
are profinite principal bundles with respect to $M \backslash G$ and $N \backslash G$, respectively, 
and the diagram 
\begin{align*}
\xymatrix{
M \backslash Y \ar[rr]^-{f_{M,N}} \ar[dr]_-{f_M} & & N \backslash Y \ar[dl]^-{f_N} \\
& X &
}
\end{align*}
commutes. 
Moreover, the map $f_{M,N}$ is $M \backslash G$-equivariant, 
i.e., $f_{M,N}(\gamma y) = \gamma f_{M,N}(y)$ for all $\gamma \in \Gamma = M \backslash G$ and $y \in N \backslash Y$.  
The action $g_M$ restricted to $\Gamma$ leaves $\Delta$ invariant 
since, for $y \in \Delta = f_{M,N}^{-1}(s(X))$ and $\gamma \in \Gamma$, we have 
$f_{M,N}(\gamma y) = f_{M,N}(y) \in s(X)$ as $N$ acts trivially on $N \backslash Y$.

\begin{lemma}\label{subbundle} 
The restriction $f_M|_{\Delta} \colon \Delta\to X$ equipped with the group action $g_M|_{\Gamma\times\Delta}:\Gamma\times\Delta\to\Delta$ is a profinite principal $\Gamma$-bundle over $X$. 
\end{lemma}
\begin{proof} 
Since $f_M|_{\Delta}$ is the restriction of a continuous map, it is continuous, and for similar reasons $g_M|_{\Gamma\times\Delta}$ is continuous and $\Gamma$-equivariant, too. 
Since $X$ is compact and $N\backslash Y$ is Hausdorff, 
the image $s(X)$ is closed by \cite[I \S 9.4, Corollary 2 on page 87]{BourbakiTop}.  
As $f_{M,N}$ is continuous, the preimage $\Delta$ is also closed. 
Therefore, it is a closed subspace of the profinite space $M\backslash Y$, and hence $\Delta$ is a profinite space. 
Finally, for every $x\in X$, since $f_M^{-1}(x)$ is an $M \backslash G$-orbit, 
the fibre $(f_M|_{\Delta})^{-1}(x)$ is $\Gamma$-equivariantly bijective to $\Gamma$  
and hence it is a $\Gamma$-orbit. 
\end{proof}

Now we can finish the proof of Theorem \ref{section}. 
By Lemma \ref{subbundle} and Proposition \ref{finite-g}, 
there is a section $r \colon X\to\Delta$. 
Let $r$ also denote the composition of this map with the inclusion map $\Delta\to M\backslash Y$ by slight abuse of notation. 
Then $(M,r)\in\mathcal S$ and we now show that $(M,r)\geq(N,s)$.   
For every $x \in X$, we have $f_N(f_{M,N}(r(x))) = x$ 
since $f_N \circ f_{M,N} = f_M$ and $r$ is a section of $f_M$. 
But, by definition of $\Delta = f_{M,N}^{-1}(s(X))$, 
$f_{M,N}(r(x))$ is an element in $s(X)$, 
i.e., there is an $x' \in X$ such that $s(x') = f_{M,N}(r(x))$. 
If $x' \ne x$, then $f_N(s(x'))=x'$ since $s$ is a section of $f_N$. 
Hence we must have $x'=x$ and $s(x) = f_{M,N}(r(x))$, 
i.e., $f_{M,N} \circ r = s$. 
This shows $(M,r)\neq(N,s)$ in $\mathcal S$. 
However, $M$ is a proper subgroup of $N$ and hence $(M,r) \ne (N,s)$. 
This contradicts the maximality of $(N,s)$, so $N=\{1\}$. 
This completes the proof of Theorem \ref{section}. 
\end{proof}


\section{Maximal pro-$2$ quotients of real projective groups}\label{sec:max_2_quotients}

The goal of this section is to prove Theorem \ref{foximaxi}. 
We first introduce the following type of embedding problems: 

\begin{defn}\label{4.1} 
Let $G$ be a profinite group. An embedding problem for $G$: 
$$\xymatrix{
 & G \ar[d]^-{\phi}\ar@{.>}[ld]_-{\widetilde{\phi}}
\\
B\ar[r]_-{\alpha} & A}$$
is a {\it $2$-embedding problem} if both $A$ and $B$ are $2$-groups.
\end{defn}


\begin{prop}\label{real2real} 
Let $G$ be a pro-$2$ group such that every real $2$-embedding problem over $G$ has a solution. 
Then every real embedding problem over $G$ has a solution, too.
\end{prop}

To prove Proposition \ref{real2real} we need the following group-theoretical 

\begin{lemma}\label{sylow} 
Let $f \colon C\to D$ be a surjective homomorphism of finite groups such that $D$ is a $2$-group. 
Let $P\subseteq C$ be a $2$-Sylow subgroup and let $x\in C$ be a $2$-torsion element. Then
\begin{enumerate}
\item[$(i)$] The restriction $f|_P \colon P\to D$ is surjective. 
\item[$(ii)$] There is an $h\in C$ such that $h^{-1}xh\in P$ and $f(x)=f(h^{-1}xh)$.
\end{enumerate}
\end{lemma}
\begin{proof} 
Let $2^b$ denote the order of $D$. 
Let $N$ be the kernel of $f$ and write the order of $N$ as $2^ar$ where $r$ is not divisible by $2$. 
Then the order of $C$ is $2^{a+b}r$, so the order of $P$ is $2^{a+b}$, while the order of $P\cap N$ is at most $2^a$. 
Since the kernel of the restriction $f|_P \colon P\to D$ is $P\cap N$, we get that the image of $f|_P$ is at least $2^b$. 
Therefore $f|_P$ is surjective. 
This proves $(i)$. 

Since the order of the subgroup generated by $x$ divides $2$, 
there is a  $t\in C$ such that $t^{-1}xt\in P$ by the second Sylow theorem. 
By part $(i)$, 
there is a $v\in P$ such that $f(v)=f(t)$. 
Set $h=tv^{-1}$. Then
$$h^{-1}xh=v(t^{-1}xt)v^{-1}\in vPv^{-1}=P,$$
since $v^{-1}\in P$. 
Moreover,
\[
f(h^{-1}xh)=f(v)f(t)^{-1}f(x)f(t)f(v)^{-1}=f(x)
\]
using that $f$ is a homomorphism and $f(v)=f(t)$. 
This proves $(ii)$. 
\end{proof}

\begin{proof}[Proof of Proposition \ref{real2real}] 
Let 
$$\xymatrix{
 & G \ar[d]^-{\phi}\ar@{.>}[ld]_-{\widetilde{\phi}}
\\
B\ar[r]_-{\alpha} & A}$$
be a real embedding problem $\mathbf E$ for $G$. 
Let $H\subseteq A$ be the image of $\phi$. 
Since $G$ is a pro-$2$ group, $H$ is a $2$-group. 
Let $C\subseteq B$ be the preimage of $H$ with respect to $\alpha$ and let $P\subseteq C$ be a $2$-Sylow subgroup. 
By part $(i)$ of Lemma \ref{sylow} the restriction $\alpha|_P \colon P \to H$ is surjective. 
Clearly,
$$\xymatrix{
 & G \ar[d]^-{\phi}\ar@{.>}[ld]_-{\widetilde{\phi}}
\\
P\ar[r]_-{\alpha|_P} & H}$$
is a $2$-embedding problem $\mathbf F$ for $G$ such that if it has a solution then $\mathbf E$ also has a solution. 
Therefore it will be enough to show that $\mathbf F$ is real because of our assumptions on $G$. 
Let $x\in G$ be an involution. 
By assumption there is a $2$-torsion element $g\in C$ such that $\alpha(g)=\phi(x)$. 
Then there is a $y\in P$ which is conjugate to $g$ in $C$ such that $\alpha(y)=\phi(x)$ by part $(ii)$ of Lemma \ref{sylow}. 
Since $y$ is conjugate to a $2$-torsion element, it is also $2$-torsion. 
So $\mathbf F$ is real. 
\end{proof}



\begin{thm}\label{foximaxi} 
Let $G$ be a pro-$2$ group. Then the following are equivalent:
\begin{enumerate}
\item[$(i)$] $G$ is real projective. 
\item[$(ii)$] $G$ is isomorphic to the maximal pro-$2$ quotient of a real projective group. 
\end{enumerate}
\end{thm}
\begin{proof} 
Since the maximal pro-$2$ quotient of a pro-$2$ group is the group itself, clearly $(i)$ implies $(ii)$. 
Now let $G$ be a pro-$2$ group which  satisfies $(ii)$. 
We start the proof of the other implication by showing that every real embedding problem for $G$ has a solution. 
By Proposition \ref{real2real} we need to show that any real $2$-embedding problem $\mathbf E$:
$$\xymatrix{
 & G \ar[d]^-{\phi}\ar@{.>}[ld]_-{\widetilde{\phi}}
\\
B\ar[r]_-{\alpha} & A}$$
has a solution. 
Recall that a field $K$ is called {\it pseudo real closed} 
if every absolutely irreducible variety defined over $K$ which has a simple $K_<$-rational point for every ordering $<$ on $K$ has a $K$-rational point, 
where $K_<$ denotes the real closure of the ordered field $(K,<)$. 
By the assumption, there is a real projective group $\Gamma$ such that $G$ is isomorphic to the maximal pro-$2$ quotient of $\Gamma$. 
By the work of Haran--Jarden in \cite[Theorem 10.4 on page 487]{HJ},  
we can then choose a pseudo real closed field $K$ such that $\Gamma$ is the absolute Galois group of $K$.  
Let $q \colon \Gamma \to G$ be the corresponding quotient homomorphism. 
We claim that
\[
\xymatrix{
 & \Gamma \ar[d]^-{\phi\circ q}\ar@{.>}[ld]_-{\widetilde{\phi}}
\\
B\ar[r]_-{\alpha} & A}
\]
is a real embedding problem which we denote by $\mathbf F$. 
Indeed, let $x\in\Gamma$ be an involution such that $\phi\circ q(x)$ is also an involution. 
Then $q(x)\in G$ is also an involution, so there is an involution $g\in B$ such that
$\alpha(g)=\phi(q(x))=(\phi\circ q)(x)$. 
So $\mathbf F$ is real. 
Since $\Gamma$ is real projective as the absolute Galois group of a pseudo real closed field, 
the embedding problem $\mathbf F$ has a solution $\widetilde{\phi} \colon \Gamma\to B$. 
But $B$ is a $2$-group, so $\widetilde{\phi}$ is the composition of $q$ and a continuous homomorphism $G\to B$. 
The latter is a solution to $\mathbf E$. 
To finish the proof of Theorem \ref{foximaxi} we need the following notation and lemma. 

\begin{notn} 
For every profinite group $G$, 
let $G_2$ denote its maximal pro-$2$ quotient and let $t_G \colon G\to G_2$ denote the quotient map. 
This assignment is functorial, that is, for every homomorphism $h \colon G\to H$ of profinite groups there is a unique homomorphism $h_2 \colon G_2\to H_2$ such that the diagram:
$$\xymatrix{
G  \ar[r]^-{h} \ar[d]_-{t_G} & H \ar[d]^-{t_H}
\\ G_2\ar[r]_-{h_2} & H_2}$$
is commutative. 
\end{notn}

\begin{lemma}\label{indi} 
Let $G$ be a profinite group, let $H\subseteq G_2$ be an open subgroup, let $I$ be the preimage $t_G^{-1}(H)\subseteq G$, 
and let $h \colon I\to H$ denote the restriction of $t_G$ onto $I$. 
Then $h_2 \colon I_2\to H_2=H$ is an isomorphism.
\end{lemma}
\begin{proof} 
Since $h$ is surjective, the map $h_2$ is also surjective, so we only need to show that it is injective, too. 
Let $1\neq\gamma\in I_2$ be arbitrary. 
Then there is an open normal subgroup $U\subseteq I_2$ such that $\gamma\not\in U$. 
Since $t_I$ is continuous, the preimage $t_I^{-1}(U)$ is an open subgroup of $2$-power index in $I$. 
Since $I$ is an open subgroup of $2$-power index in $G$, we get that $t_I^{-1}(U)$ is an open subgroup of $2$-power index in $G$, too. 
Set $N \coloneqq \bigcap_{\delta\in G}\delta^{-1}t_I^{-1}(U)\delta$. 
Clearly $N$ is a normal subgroup. 
Since $\delta^{-1}t_I^{-1}(U)\delta$ only depends on the coset $t_I^{-1}(U)\delta$, of which there are only finitely many, we get that $N$ is a finite intersection of open subgroups of $2$-power index in $G$, so it is also an open  subgroup of $2$-power index in $G$. 
Therefore the subgroup $t_I(N)\subseteq I_2$ is the preimage of $t_G(N)\subseteq G_2$, but clearly $\gamma\not\in t_I(N)$, and hence $h_2(\gamma)\neq1$.
\end{proof}

Now we return to the proof of Theorem \ref{foximaxi}. 
By Definition \ref{def:real_projective_group}, it remains to show that $G = \Gamma_2$ contains an open subgroup without $2$-torsion. 
Let $K$ be the field introduced above, and let
$\Delta\subseteq\Gamma$ be the open subgroup corresponding to the finite extension $K(\sqrt{-1})/K$. 
As $\Delta$ is isomorphic to the absolute Galois group of $K(\sqrt{-1})$, 
it has cohomological dimension at most $1$ by \cite[Corollary 2.4]{PQ}. 
By the Rost--Voevodsky norm residue theorem \cite{Voe}  
(formerly known as the Milnor conjecture), 
the pull-back map $\Hb(\Delta_2,\mathbb Z/2\mathbb Z)\to \Hb(\Delta,\mathbb Z/2\mathbb Z)$ induced by the quotient map $t_{\Delta} \colon \Delta\to\Delta_2$ is an isomorphism. 
Therefore, $\Delta_2$ also has cohomological dimension at most $1$.  
Hence, by \cite[Proposition 14 on page 19]{Se}, any closed subgroup of $\Delta_2$ has cohomological dimension at most $1$. 
Since a finite subgroup would have infinite cohomological dimension, $\Delta_2$ is torsion-free. 
Moreover, $\Delta$ has index dividing two in $\Gamma$, so $\Delta_2$ is isomorphic
 to the kernel of the homomorphism $\Gamma_2\to\mathbb Z/2\mathbb Z$ 
 corresponding to the homomorphism $\Gamma \to \mathrm{Gal}(K(\sqrt{-1})/K) \subseteq \mathbb Z/2\mathbb Z$ by Lemma \ref{indi}. 
Therefore, 
$\Delta_2$ is an open, torsion-free subgroup of $\Gamma_2$. 
This finishes the proof of Theorem \ref{foximaxi}. 
\end{proof}


\section{Pro-$2$ real projective groups versus quasi-Boolean groups}\label{sec:pro2_real_vs_quasi_Boolean}

The goal of this section is to prove Theorems \ref{qb_rp} and \ref{rp_qb}. 
We begin with the following recollection and notation. 

\begin{rem} 
The free product  $G_1*_pG_2$ of pro-$p$ groups $G_1$ and $G_2$ is the coproduct of $G_1$ and $G_2$ in the category of pro-$p$ groups, that is, it has the following universal property. 
For $j=1,2$, let $\iota_j \colon G_j\to G_1*_pG_2$ be the composition 
of the natural inclusion $G_j \to G_1*G_2$ and the quotient homomorphism  $G_1*G_2\to G_1*_pG_2$. 
Then for every pro-$2$ group $G$ and for every pair of homomorphisms $f_j \colon G_j\to G$ of pro-$2$ groups, 
there is a unique homomorphism $f_1*_pf_2 \colon G_1*_pG_2\to G$ of pro-$2$ groups such that $(f_1*_pf_2)\circ\iota_j=f_j$ for $j=1,2$. 
This follows obviously from the definition when $G$ is finite, and the general case follows by taking the projective limit. 
\end{rem}

\begin{notn} 
Let $G$ be a profinite group. Let $\mathcal Y_p(G)$ denote the subset of elements of order dividing $p$ in $G$. We equip
$\mathcal Y_p(G)$ with the subspace topology. 
Let $\mathcal X_p(G)$ denote the quotient of $\mathcal Y_p(G)$ by the conjugation action of $G$. 
We equip $\mathcal X_p(G)$ with the quotient topology. 
Let $\mathcal Y_p^*(G)$ denote the complement of $1$ in $\mathcal Y_p(G)$, 
and let $\mathcal X_p^*(G)$ denote the complement of the conjugacy class of $1$ in $\mathcal X_p(G)$. 
When $p=2$ we let $\mathcal Y(G),\mathcal X(G),\mathcal Y^*(G),\mathcal X^*(G)$ denote 
$\mathcal Y_p(G),\mathcal X_p(G),\mathcal Y_p^*(G),\mathcal X_p^*(G)$, respectively. 
We note that $\mathcal Y^*(G)$ equals the set $\Inv(G)$ of involutions in $G$. 
\end{notn}

\begin{rem}\label{uki-ho} 
Let $X$ be a topological space. 
Recall from Definition \ref{def:B_of_X} the free pro-$2$ product $\mathbb B(X)$. 
We note that $\mathbb B(X)$ has the following universal property: 
Let $\iota_X \colon X\to\mathbb B(X)$ be the 
composition of the natural inclusion $X\to
\bigast_X\mathbb Z/2\mathbb Z$ and the quotient homomorphism
$\bigast_X\mathbb Z/2\mathbb Z\to\mathbb B(X)$. 
Then, for every pro-$2$ group $G$ and for every continuous map $f \colon X \to \mathcal Y(G)$, 
there is a unique homomorphism $b_f \colon \mathbb B(X)\to G$ of pro-$2$ groups such that $b_f\circ\iota_X=f$. 
This follows from the definition when $G$ is finite, and the general case follows by taking the projective limit. 
\end{rem}

\begin{notn} 
Let $f \colon X \to Y$ be a continuous map of topological spaces. 
Since $i_Y \colon Y\to\mathbb B(Y)$ is continuous and its image lies in $\mathcal Y(\mathbb B(Y))$, 
by the universal property in Remark \ref{uki-ho}  above there is a unique homomorphism $\mathbb B(f) \colon \mathbb B(X)\to\mathbb B(Y)$ of pro-$2$ groups such that $\mathbb B(f)\circ\iota_X=\iota_Y\circ f$. 
This makes the assignment $X\mapsto\mathbb B(X)$ into a functor. 
\end{notn}

\begin{prop}\label{replace} 
For every topological space $X$, the map $\mathbb B(u_X) \colon \mathbb B(X) \to \mathbb B(\widehat X)$ induced by the profinite completion $u_X \colon X\to\widehat X$ is an isomorphism of pro-$2$ groups. 
\end{prop}
\begin{proof} 
%
First we are going to show that $\mathbb B(u_X)$ is surjective. 
We recall that $\mathbb B(X)$ is a profinite space by construction. 
Hence, by $(i) \Rightarrow (iii)$ of Theorem \ref{mama}, 
$\mathbb B(X)$ is compact and Hausdorff. 
Since the continuous image of a compact space is compact by \cite[I \S 9.4,  Corollary 1 on page 87]{BourbakiTop}, 
we conclude that the image of $\mathbb B(u_X)$ is compact. 
Since $\mathbb B(\widehat X)$ is Hausdorff, the image of $\mathbb B(u_X)$ is closed
 in $\mathbb B(\widehat X)$ by \cite[I \S 9.4, Corollary 2 on page 87]{BourbakiTop}.  
Hence it will be sufficient to show that the image is dense. 
In order to do so it will be enough to prove that, for every $2$-group $G$ 
and continuous surjective homomorphism $f \colon \mathbb B(\widehat X)\to G$, 
the composition $f\circ\mathbb B(u_X)$ is surjective. 
Note that $\widehat X$ generates $\bigast_{\widehat X}\mathbb Z/2\mathbb Z$, so $f(\widehat X)$ generates $G$. 
By \cite[Lemma 1.1.7]{RZ}, the image $u_X(X)$ of $X$ is dense in $\widehat X$. 
This implies that the image $f\circ u_X(X)$ is dense in $f(\widehat X)$. 
But $G$ is finite, so it is discrete, and hence $f\circ u_X(X)$ is equal to $f(\widehat X)$. So $f\circ u_X(X)$ generates $G$, therefore $f\circ\mathbb B(u_X)$ is surjective.

Next we are going to show that $\mathbb B(u_X)$ is injective. 
In order to do so it will be sufficient to show that $\mathbb B(u_X) \circ \iota_X \colon X\to\mathbb B(\widehat X)$ 
has the universal property described in Remark \ref{uki-ho}. 
Indeed then there is a continuous homomorphism $f \colon \mathbb B(\widehat X)\to\mathbb B(X)$ such that $f\circ\mathbb B(u_X)\circ\iota_X$ is $\iota_X$, 
and hence $f\circ\mathbb B(u_X)$ is the identity of $\mathbb B(X)$. 
Now let $G$ be a  pro-$2$ group and $f \colon X\to\mathcal Y(G)$ be a continuous map. 
Since $\mathcal Y(G)$ is profinite, there is a unique continuous map $g \colon \widehat X\to\mathcal Y(G)$ such that $f=g\circ u_X$ because of the universal property of $u_X$. 
Using the universal property of $\mathbb B(\widehat X)$ we get that there is a continuous homomorphism $b_g \colon \mathbb B(\widehat X)\to G$ of pro-$2$ groups such that $b_f\circ\iota_{\widehat X}=g$. 
By the functoriality of the assignment $X \to \mathbb B(X)$,
$\iota_{\widehat X} \circ u_X = \mathbb B(u_X)\circ \iota_X$. 
Thus,
\begin{equation*}
b_f \circ \mathbb B(u_X) \circ \iota_X = b_f \circ \iota_{\widehat X} \circ u_X=g\circ u_X=f,
\end{equation*}
as desired.
\end{proof}

We recall from Definition \ref{def:Boolean_etc_intro} that we call a pro-$2$ group {\it Boolean} if it is isomorphic to $\mathbb B(X)$ for some topological space $X$. 
By Proposition \ref{replace} we can always assume that $X$ is profinite. 
We call a pro-$2$ group {\it quasi-Boolean} if it is the free product of a free pro-$2$ group and a Boolean group. 

\begin{thm}\label{qb_rp} 
Every quasi-Boolean pro-$2$ group is real projective. 
\end{thm}
\begin{proof} Let $G$ be a quasi-Boolean pro-$2$ group. 
It will be sufficient to show that $G$ is the maximal pro-$2$ quotient of a real projective group by Theorem \ref{foximaxi}. 
In order to do so, we will use a group-theoretical characterisation of real projective groups by Haran and Jarden. 
Following \cite[Definition 1.1 on page 156]{HJ2} we define: 

\begin{defn} 
A profinite group $D$ is said to be {\it real free} if it contains disjoint closed subsets $X$ and $Y$ such that $X\subseteq\mathcal Y^*(D)$, $1\in Y$, and every continuous map $\phi$ from $X\cup Y$ into a profinite group $H$ such that $\phi(x)^2=1$ for every $x\in X$ and $\phi(1)=1$ extends to a unique homomorphism of $D$ into $H$.
\end{defn}

\begin{thm}[Haran--Jarden]\label{real-free} 
A profinite group $G$ is real projective if and only if $G$ is isomorphic to a closed subgroup of a real free group. 
\end{thm}
\begin{proof} 
This claim is \cite[Theorem 3.6 on page 160]{HJ2}. 
\end{proof}

Now we return to the proof of Theorem \ref{qb_rp}.  
For every set $Y$, let $F(Y)$ denote the free pro-$2$ group on $Y$ as defined in \cite[Section 1.5 on pages 7--8]{Se}. 
By assumption, $G$ is the free product of a free pro-$2$ group $F(Y)$ 
and a Boolean group $\mathbb B(X)$ for a set $Y$ and a profinite space $X$. 
By Theorem \ref{real-free} it will be sufficient to show 
that $G$ is the maximal pro-$2$ quotient of a real free profinite group. 
Let $\widehat G$ be the free product $\bigast_Y\mathbb Z*\bigast_X\mathbb Z/2\mathbb Z$, 
i.e., the group which is freely generated by the elements of the disjoint union of $Y$ and $X$, 
subject to the relation that the elements in $X$ are involutions. 
Let $\mathcal N$ be the family of normal subgroups $N$ of $\widehat G$ such that 
\begin{enumerate}
\item[$(i)$] the quotient $\widehat G/N$ is finite,
\item[$(ii)$] the composition of the natural inclusion $Y\to\widehat G$ and the quotient homomorphism $\widehat G\to\widehat G/N$ maps all but finitely many elements of $Y$ to $1$,
\item[$(iii)$] the composition of the natural inclusion $X\to\widehat G$ and the quotient homomorphism $\widehat G\to\widehat G/N$ is continuous with respect to the discrete topology on $\widehat G/N$.
\end{enumerate}
Set
$$\overline G=\varprojlim_{N\in\mathcal N}\widehat G/N.$$
Clearly $G$ is the maximal pro-$2$ quotient of $\overline G$. 
On the other hand, $\overline G$ is a real free profinite group. 
In fact, $\overline G$ is the group $\widehat D(X,Y_+,e)$ in \cite[Lemma 1.3 on pages 156--157]{HJ2}, where $Y_+$ is the one-point compactification of $Y$ equipped with the discrete topology, 
and $e\in Y_+-Y$ is the point at infinity. 
This finishes the proof of Theorem \ref{qb_rp}. 
\end{proof}

Next, we set out to prove the converse, i.e., every real projective pro-$2$ group is quasi-Boolean. 
We begin with the following 

\begin{defn} 
Let $A$ be an abelian profinite group. 
A {\it complement} of a closed subgroup $B\subset A$ is a closed subgroup $C\subseteq A$ such that $B\cap C$ is trivial, and $B+C=A$. 
\end{defn}

The following lemma is probably very well-known, but we could not find a convenient reference. 
For its proof we recall from \cite[Definition 1.1.10]{NSW} that the {\it Pontryagin dual} 
of a Hausdorff, abelian and locally compact group $A$ is the group $A^{\vee} = \Hom_{\mathrm{cont}}(A,\R/\Z)$ of continuous group homomorphisms from $A$ to $\R/\Z$ where the latter is equipped with the quotient topology inherited from $\R$. 
The group $A^{\vee}$ is considered as a topological group with the compact-open topology. 
We recall that Pontryagin duality states that the canonical homomorphism $A \to (A^{\vee})^{\vee}$ is an isomorphism 
(see \cite[Theorem 1.1.11]{NSW} and the references therein). 
In particular, for an abelian profinite group $G$, 
the Pontryagin dual $G^{\vee}$ is a discrete abelian torsion group. 

\begin{lemma}\label{mod_p} Let $G$ be a $p$-torsion abelian profinite group. Then the following holds:
\begin{enumerate}
\item[$(i)$]  there is an isomorphism $G\cong\mathbb F_p^X$ for some set $X$, 
\item[$(ii)$] every closed subgroup of $G$ has a complement. 
\end{enumerate}
\end{lemma}
\begin{proof} 
The Pontryagin dual of $G$ is a discrete $\mathbb F_p$-linear vector space $V$ since $G$ is compact. 
Note that every isomorphism between discrete $\mathbb F_p$-vector spaces is automatically a homeomorphism, so $V$ is isomorphic to the direct sum $\mathbb F_p^{\oplus X}$, equipped with the discrete topology, as a topological group for some set $X$. 
The Pontryagin dual of $\mathbb F_p^{\oplus X}$ is $\mathbb F_p^X$, which is isomorphic to $G$ by Pontryagin duality. 
So claim $(i)$ holds.

Assertion $(ii)$ is equivalent to the following claim: let $i \colon B\to A$ be a monomorphism of $p$-torsion abelian profinite groups. 
Then there is a morphism $j \colon A\to B$ such that $j\circ i=\mathrm{id}_B$. 
By Pontryagin duality, it is equivalent to the following claim: let $p \colon V\to U$ be an epimorphism of $\mathbb F_p$-linear vector spaces. 
Then there is a morphism $r \colon U\to V$ such that $p\circ r=\mathrm{id}_U$. 
The latter is well-known. 
\end{proof}

For later purposes, we record the following fact: 

\begin{thm}\label{sistar} 
Let $G$ be a quasi-Boolean pro-$2$ group which is the free product of a free pro-$2$ group $F$ and a Boolean group $\mathbb B(X)$ for a profinite space $X$. 
Let $j_X \colon X\to \mathcal X(\mathbb B(X))$ denote the composition of the map $i_X \colon X  \to\mathcal Y(\mathbb B(X))$, 
the inclusion $\mathcal Y(\mathbb B(X))\subseteq\mathcal Y(G)$, and the quotient map $\mathcal Y(G)\to\mathcal X(G)$. 
Then the map $j_X$ is a homeomorphism onto $\mathcal X^*(G)$. 
\end{thm}

To prove the theorem we first show the following 

\begin{lemma}\label{lemma:X*(G)profinite}
For a real projective group $G$, the space $\mathcal X^*(G)$ is compact and totally separated. 
\end{lemma}
\begin{proof}
Since $G$ is real projective, the subset $\mathcal Y^*(G) = \Inv(G)$ of involutions is closed in $G$. 
Hence $\Inv(G)$ is compact and totally separated by $(i) \Rightarrow (iv)$ of Theorem \ref{mama}.    
This shows that the quotient $\mathcal X^*(G)$ of conjugacy classes is also compact by \cite[I \S 9.4, Theorem 2 on page 87]{BourbakiTop}. 
Now let $x,y$ be two distinct involutions in $G$ which are not conjugate. 
Since $\Inv(G)$ is totally separated, there is an open and closed subset $U \subset \Inv(G)$ such that 
$x\in U$ and $y \notin U$. 
By definition of the subspace topology and since the open normal subgroups form a basis for the topology on $G$, 
there is an open normal subgroup $\widetilde{U}$ of $G$ such that $U = \Inv(G) \cap \widetilde{U}$.  
Since $\widetilde{U}$ is conjugation-invariant, both $U$ and $\Inv(G) - U$ are conjugation-invariant. 
This implies that the image of $U$ under the quotient map $\Inv(G) \to \mathcal X^*(G)$ is open   
and closed and contains the conjugacy class of $x$ but not the class of $y$. 
This shows that $\mathcal X^*(G)$ is totally separated. 
%
%
\end{proof}

\begin{proof}[Proof of Theorem \ref{sistar}] 
By Lemma \ref{lemma:X*(G)profinite}, $\mathcal X^*(G)$ is totally separated and hence Hausdorff. 
Since $X$ is compact, 
it will therefore be sufficient to show that $j_X$ maps $X$ onto $\mathcal X^*(G)$ bijectively by \cite[I \S 9.4, Corollary 2 on page 87]{BourbakiTop}. 
Let $x,y\in X$ be two distinct elements. 
Then there is a continuous map $s \colon X\to \mathbb Z/2\mathbb Z$ such that $s(x)\neq s(y)$. 
Let $\overline s \colon G\to\mathbb Z/2\mathbb Z$ be the unique homomorphism such that the restriction of $\overline s$ onto $F$ is trivial, and onto $\mathbb B(X)$ is $b_s$ (as defined in Remark \ref{uki-ho}). 
Under $\overline s$ the images of $i_X(x)$ and $i_X(y)$ are not conjugate, so they are not conjugate in $G$ either. 
Therefore $j_X$ is injective. 
Since we may guarantee that $s(x)$ is not a unit, we get that $j_X$ maps into $\mathcal X^*(G)$, too. 

Now assume that there is a $y\in\mathcal X^*(G)$ which is not in the image of $j_X$. 
Since $X$ is compact and $\mathcal X^*(G)$ is Hausdorff, 
the image of $j_X$ is closed by \cite[I \S 9.4, Corollary 2 on page 87]{BourbakiTop}. 
Thus, by Lemma \ref{sep2}, which applies to $\mathcal X^*(G)$ by Lemma \ref{lemma:X*(G)profinite},  
there is a continuous function $r \colon \mathcal X^*(G)\to\mathbb Z/2\mathbb Z$ such that $r\circ j_X$ is zero, and $r(y)$ is non-zero. 
By Theorem \ref{qb_rp}, the pro-$2$ group $G$ is real projective. 
Hence, by Scheiderer's theorem \cite[Theorem 2.11]{PQ}, 
there is a continuous homomorphism $\overline r \colon G\to\mathbb Z/2\mathbb Z$ whose image under the map
$$\pi_1\colon H^1(G,\mathbb Z/2\mathbb Z)\to C(\mathcal X^*(G),\mathbb Z/2\mathbb Z)$$
is $r$ 
where $C(\mathcal X^*(G),\Z/2\Z)$ denotes 
the ring of continuous functions from $\mathcal X^*(G)$ to $\Z/2\Z$,
where the latter is equipped with the discrete topology.  
Then the restriction of $\overline r$ onto $i_X(X)$ is zero, so by the universal property of $\mathbb B(X)$ the restriction of $\overline r$ onto $\mathbb B(X)$ is also zero, and hence this homomorphism factors through the surjective homomorphism $p_1 \colon G\to F$ supplied by the universal property of free pro-$2$ products. 
But $F$ is torsion-free (see \cite[Corollary 4 on page 31]{Se}), so $p_1$ is zero on any involution $\overline y$ in the conjugacy class $y$. 
Therefore $r(y)=\overline r(\overline y)$ is zero, which is a contradiction. 
\end{proof}

\begin{thm}\label{have_sec} 
Let $G$ be isomorphic to the absolute Galois group of a field $K$. 
Then there is a continuous section $s \colon \mathcal X^*(G)\to\mathcal Y^*(G)$. 
\end{thm}
\begin{proof} Let $H\subseteq G$ be the open subgroup corresponding to the finite extension $K(\sqrt{-1})/K$. 
Then $H$ is isomorphic to the absolute Galois group of $K(\sqrt{-1})$, 
and hence it is torsion-free by the Artin--Schreier theorem. 
In particular, if $H=G$ then $\mathcal X^*(G)$ is empty and the claim is trivially true. 
Otherwise, $H$ has index two in $G$. 
Then the theorem follows from Proposition \ref{prop:to_prove_thm_have_sec} below and Theorem \ref{section}.
\end{proof}

\begin{prop}\label{prop:to_prove_thm_have_sec} 
With the above assumptions, the map $\mathcal Y^*(G)\to\mathcal X^*(G)$ is a profinite principal $H$-bundle with respect to the conjugation action of $H$ on
$\mathcal Y^*(G)$, and $\mathcal X^*(G)$ is Hausdorff. 
\end{prop}
\begin{proof} 
Clearly $\mathcal Y(G)$ is closed in $G$. 
Since $H$ is torsion-free, the subset $\mathcal Y^*(G)$ is 
the intersection of $\mathcal Y(G)$ and the complement of the open $H$ in $G$, so it is closed, too. 
Let $C_G(y)$ denote the centrailzer of $y$ in $G$.  
Since $G$ is profinite, we get that $\mathcal Y^*(G)$ is profinite. 
As $G$ is isomorphic to an absolute Galois group, for every $y\in\mathcal Y^*(G)$ we have $H\cap C_G(y)=\{1\}$, and hence $H$ acts freely on $\mathcal Y^*(G)$. 
This action is also clearly continuous. 

We claim that every $x\in G$ conjugate to $y$ is already conjugate under $H$. 
Indeed, let $z\in G$ be such that $z^{-1}xz=y$. 
Since $H$ has index $2$ and $y\not\in H$, we have $G=H\cup Hy$. If $z\in H$ the claim is clearly true. 
Otherwise $z=hy$ for some $h\in H$, and hence $x=zyz^{-1}=
hyyy^{-1}h^{-1}=hyh^{-1}$, so $x$ is conjugate to $y$ under $H$ in this case, too. 
So the map $\mathcal Y^*(G)\to\mathcal X^*(G)$ is the quotient map with respect to the action of $H$. 
Since $\mathcal Y^*(G)$ is open in $\mathcal Y(G)$, 
the subspace topology on $\mathcal X^*(G)\subset\mathcal X(G)$ is the quotient topology with respect to the map in the claim. 
Therefore, $\mathcal X^*(G)$ is Hausdorff by Proposition \ref{free}. 
\end{proof}

\begin{cor}\label{have_sec2} 
Let $G$ be a real projective profinite group. Then there is a continuous section $s \colon \mathcal X^*(G)\to\mathcal Y^*(G)$. 
\end{cor}
\begin{proof} 
By \cite[Theorem 10.4 on page 487]{HJ}, every real projective profinite group is isomorphic to the absolute Galois group of a pseudo real closed field. 
The claim then follows at once from Theorem \ref{have_sec}. 
\end{proof}

\begin{rem}
Note that Corollary \ref{have_sec2} is also part $(a)$ of \cite[Lemma 3.5 on page 160]{HJ2}. 
We think, however, that our proof is more conceptual and derives a similar claim for a much larger class of groups. 
\end{rem}

\begin{notn}\label{notn:lower_star}
For every pro-$2$ group $G$, let $G_*$ denote the {\it maximal abelian $2$-torsion quotient of $G$} 
and, for every homomorphism $b \colon G \to H$ of pro-$2$ groups, 
let $b_* \colon G_* \to H_*$ denote the homomorphism induced by $b$. 
\end{notn}


\begin{thm}\label{rp_qb} 
Let $G$ be a real projective pro-$2$ group. 
Then $G$ is quasi-Boolean. 
\end{thm}


For the proof of the theorem we recall the following facts from 
\cite[Proposition 24 on page 30]{Se} for the special case $p=2$:  

\begin{prop}
\label{prop:Serre_Prop24}
Let $G$ be a pro-$2$ group and $I$ a set. 
Let 
\[
\theta \colon H^1(G,\Z/2\Z) \to (\Z/2\Z)^I
\]
be a homomorphism. 
Then: 
\begin{itemize}
\item[(a)] There exists a morphism $f \colon F(I) \to G$ such that $\theta = H^1(f)$ where $F(I)$ denotes the free pro-$2$ group on $I$. 
\item[(b)] If $\theta$ is injective, the morphism $f$ is surjective. \qed
\end{itemize}
\end{prop}

\begin{proof}[Proof of Theorem \ref{rp_qb}]  
Let $X$ denote $\mathcal X^*(G)$, and let $s \colon X\to \mathcal Y^*(G)$ be the section furnished by Theorem \ref{have_sec}. 
We let $B$ denote the image of the homomorphism
$(b_s)_* \colon \mathbb B(X)_*\to G_*$ 
where $b_s \colon \mathbb B(X) \to \mathcal Y^*(G)\subset G$ is 
the homomorphism defined in Remark \ref{uki-ho}. 
Since $\mathbb B(X)_*$ is compact and $G_*$ is Hausdorff, $B$ is closed in $G_*$ by \cite[I \S 9.4, Corollary 2 on page 87]{BourbakiTop}.  
Let $A\subseteq G_*$ be a complement of $B$, which exists by part $(ii)$ of Lemma \ref{mod_p}. 
By part $(i)$ of Lemma \ref{mod_p}, there is a set $Y$ such that $A\cong (\Z/2\Z)^Y$. 
Since $F(Y)_*\cong (\Z/2\Z)^Y$, 
by part (a) of Proposition \ref{prop:Serre_Prop24} 
and Pontryagin duality,  
there is a continuous homomorphism $h \colon F(Y)\to G$ such that $h_* \colon F(Y)_*\to G_*$ maps $F(Y)_*$ isomorphically onto $A$.   
We set $P=F(Y)*_2\mathbb B(X)$ and let $\alpha \colon P\to G$ be the homomorphism $h*_2b_s$. 
Since $(G_1*_2G_2)_*\cong(G_1)_*\oplus(G_2)_*$
for every pair of pro-$2$ groups $G_1$ and $G_2$, we get that
$\alpha_* \colon P_*\to G_*$ is an isomorphism. 
Hence, by part (b) of Proposition \ref{prop:Serre_Prop24} 
and Pontryagin duality, the map $\alpha$ is surjective.  
To finish the proof we need the following 

\begin{lemma}\label{have_sec3} 
Let $\alpha \colon P\to G$ be a continuous surjective homomorphism of real projective profinite groups, 
and let $X\subseteq\mathcal Y^*(P)$ be a system of representatives of $\mathcal X^*(P)$. 
If $\alpha$ maps $X$ bijectively onto a system of representatives of $\mathcal X^*(G)$, 
then there is a continuous injective homomorphism $\gamma \colon G\to P$ such that
$\alpha\circ\gamma=\mathrm{id}_G$. 
\end{lemma}
\begin{proof} This claim is part (b) of \cite[Lemma 3.5 on page 160]{HJ2}. 
The proof relies on the projectivity of the Artin--Schreier structures attached to real projective groups (see \cite[Proposition 7.7 on page 473]{HJ}). 
\end{proof}

We return to the proof of Theorem \ref{rp_qb}. 
By Theorem \ref{qb_rp}, the quasi-Boolean group $P$ is real projective, 
while by Theorem \ref{sistar} the subset $X\subseteq P$ is a system of representatives of $\mathcal X^*(P)$ 
mapped bijectively onto a system of representatives of $\mathcal X^*(G)$, so the conditions in Lemma \ref{have_sec3} above hold. 
Therefore there is a continuous injective homomorphism $\gamma \colon G\to P$ such that $\alpha\circ\gamma=\mathrm{id}_G$. 
Then the composition 
$\alpha_*\circ\gamma_* = (\alpha\circ\gamma)_* = (\mathrm{id}_G)_* = \mathrm{id}_{G_*}$  
is the identity, and $\alpha_*$ is an isomorphism, so $\gamma_*$ is an isomorphism, too. 
Therefore, $\gamma$ is surjective by part (b) of Proposition \ref{prop:Serre_Prop24} 
and Pontryagin duality. 
So $\gamma$ is an isomorphism, and hence $G$ is quasi-Boolean.  
This finishes the proof of Theorem \ref{rp_qb}. 
\end{proof}


\section{Quillen's theorems and their consequences}\label{sec:Quillen_consequences}

The goal of this section is to prove Corollary \ref{local-global}.  

\begin{defn} 
Following Quillen we say that a homomorphism $R\to S$ of graded anti-commutative rings is {\it finite} if $S$ is a finitely generated module over $R$. 
\end{defn}

This definition is a bit ambiguous since it does not specify whether we consider $S$ as a left $R$-module or a right $R$-module. 
However, note that $S$ is a finitely generated left $R$-module if and only if it is a finitely generated right $R$-module. 
Indeed if $S$ is a finitely generated left $R$-module then it is also generated by a finite set $H\subset S$ of homogeneous elements. 
But the left $R$-module generated by $H$ is the same as the right $R$-module generated by $H$ since $S$ is anti-commutative. 
Therefore $S$ is also finitely generated as a right $R$-module. The converse could be proved similarly.

\begin{thm}[Quillen]\label{fingen} 
Let $G$ be a pro-$p$ group and let $H\subseteq G$ be a finite subgroup. 
Then the homomorphism $\Hb(G,\mathbb Z/p\Z)\to \Hb(H,\mathbb Z/p\Z)$ is finite. 
\end{thm}
\begin{proof} 
Recall that every finite subset of a Hausdorff space is closed by \cite[I \S 8.2, Proposition 4 on page 77]{BourbakiTop}. 
Since $G$ is Hausdorff, we conclude that $H$ is a closed subgroup. 
Hence the homomorphism $\Hb(G,\mathbb Z/p\Z) \to \Hb(H,\mathbb Z/p\Z)$ is well-defined. 
The analog of the theorem was proved by Quillen when $G$ is finite \cite[Corollary 2.4 on page 555]{Qu}, 
and we now show that the general case follows as an easy corollary. 
Indeed let $N\triangleleft G$ be an open normal subgroup such that the restriction of the quotient map $q \colon G\to G/N$ to $H$ is injective. 
Then the homomorphism $\Hb(G/N,\mathbb Z/p\Z) \to \Hb(H,\mathbb Z/p\Z)$ induced by the composition of the inclusion map $H\to G$ and $q$ is finite by the above. 
Since this homomorphism factors through $\Hb(G,\mathbb Z/p\Z) \to \Hb(H,\mathbb Z/p\Z)$, the latter is also finite. 
\end{proof}

\begin{cor}\label{negligible} 
Let $G$ be a pro-$p$ group and let $H\subseteq G$ be a subgroup of order $p$. 
Then the homomorphism $H^n(G,\mathbb Z/p\Z)\to H^n(H,\mathbb Z/p\Z)$ is non-zero for infinitely many $n$. 
\end{cor}
\begin{proof} 
Assume that the claim is false and there is a natural number $d$ 
such that the image of $H^n(G,\mathbb Z/p\Z)\to H^n(H,\mathbb Z/p\Z)$ is zero for $n\geq d$. 
By Theorem \ref{fingen}, we can find a finite subset $S\subset \Hb(H,\mathbb Z/p)$ of homogeneous elements 
which generate $\Hb(H,\mathbb Z/p\Z)$ as a $\Hb(G,\mathbb Z/p\Z)$-module. 
Let $d'$ be the maximal degree of the elements of $S$. 
Then $H^n(H,\mathbb Z/p\Z)=0$ for every $n\geq d+d'$. 
But, since $H$ has order $p$, $H^n(H,\mathbb Z/p\Z)\neq 0$ for every $n$ which is a contradiction. 
\end{proof}

We now recall the following types of algebras from \cite{PQ}. 

\begin{defn}\label{def:boolean_dual_sum} 
We call an $\F_2$-algebra $\Bb = \bigoplus_{i\ge 0} B^i$ a {\it graded Boolean algebra} if $B^0=\mathbb F_2$ and 
there is a Boolean ring $B$ such that, for every $i \geq 1$, we have $B^i=B$, 
and, for every pair $i,j \ge 1$, the multiplication $B^i \times B^j\to B^{i+j}$ is the multiplication in the ring $B=B^i=B^j=B^{i+j}$. 
We call an $\F_2$-algebra $\Db = \bigoplus_{i\ge 0} D^i$ a {\it dual algebra} if $D^0=\F_2$, and $D^i=0$ for $i \geq 2$. 
The {\it connected sum} $\Db \sqcap \Bb$ is the graded $\F_2$-algebra with $(\Db \sqcap \Bb)^0=\F_2$, $(\Db \sqcap \Bb)^i= D^i \oplus B^i$ for $i\geq 1$ and multiplication $D^1B^i$ and $B^iD^1$ is set to be zero for all $i \geq 1$. 
\end{defn}

\begin{rem}
In \cite{PQ} we show that dual and graded Boolean algebras are Koszul algebras. 
In particular, they are quadratic algebras and their connected sum is their direct sum as quadratic algebras. 
\end{rem}

The results in \cite{PQ} are the motivation for the following terminology which we recall from the introduction. 

\begin{defn} 
We say that a pro-$2$ group is a {\it cohomologically Boolean group} if its mod $2$ cohomology is a graded Boolean algebra. 
We say that a pro-$2$ group is a {\it cohomologically quasi-Boolean group} if its mod $2$ cohomology is the connected sum of a dual algebra and a graded Boolean algebra. 
\end{defn}

For every commutative ring $R$, let $\mathcal N(R)$ denote the nilradical of $R$. 

\begin{lemma}\label{7.6} 
Let $G$ be a cohomologically quasi-Boolean pro-$2$ group. 
Then the quotient $\Hb(G,\mathbb Z/2\mathbb Z)/\mathcal N(\Hb(G,\mathbb Z/2\mathbb Z))$ is a graded Boolean algebra. 
\end{lemma}
\begin{proof} 
Let $\Hb(G,\mathbb Z/2\mathbb Z)$ be the connected sum of a dual algebra $\Db$ and a graded Boolean algebra $\Bb$. 
Since all elements of $D^1$ are nilpotent while no element of $\Bb$ is, 
we see that $D^1$ is the nilradical of the ring $\Hb(G,\mathbb Z/2\mathbb Z) = \Db \sqcap \Bb$. 
Hence the quotient $\Hb(G,\mathbb Z/2\mathbb Z)/\mathcal N(\Hb(G,\mathbb Z/2\mathbb Z))$ is isomorphic to $\Bb$. 
\end{proof}

\begin{defn} 
Let $G$ be a cohomologically quasi-Boolean pro-$2$ group. 
Let $B$ denote, up to isomorphism, the unique Boolean ring 
such that the associated graded Boolean algebra $\Bb$ is isomorphic 
 to $\Hb(G,\mathbb Z/2\mathbb Z)/\mathcal N(\Hb(G,\Z/2\Z))$ 
 which exists by Lemma \ref{7.6}.   
We say that a continuous homomorphism $k \colon G \to \Z/2\Z$ is a {\it quasi-canonical homomorphism} if the image of the associated cohomology class $k\in H^1(G,\Z/2\Z)$ under the quotient map
\[
\Hb(G,\Z/2\Z)\to \Hb(G,\Z/2\Z)/\mathcal N(\Hb(G,\Z/2\Z))
\]
is the unit of $B$.
\end{defn}

\begin{rem}\label{6.8} 
Quasi-canonical homomorphisms $k\in H^1(G,\mathbb Z/2\mathbb Z)$ can be characterised by the following property: 
for every $n>0$ and $c\in H^n(G,\mathbb Z/2\mathbb Z)$, we have $c^2=c \cup k^n$. 
To prove this assertion, we write $\Hb = \Hb(G,\mathbb Z/2\mathbb Z)$ as the connected sum of $\Db$ and $\Bb$.
Then $k\in H^1(G,\mathbb Z/2\mathbb Z)$ is quasi-canonical if and only if its image under the quotient map $\Hb/D^{\bullet >0} \to \Bb$ 
is the identity map since $D^{\bullet >0}$ is the nilradical of $\Hb$ as pointed out in the proof of Lemma \ref{7.6}. 
The identity $c^2=c \cup k^n$ holds for all elements of $\Bb$,   
and this identity is equivalent to $k$ being the identity in $B$ by definition of the multiplication in $\Bb$ in Definition \ref{def:boolean_dual_sum}. 
The same identity trivially holds for all $c$ in $D^{\bullet>0}$  and any $k$. 
\end{rem}

\begin{defn} 
An {\it elementary $p$-group} $H$ is a group isomorphic to $(\mathbb Z/p\Z)^n$ for some $n$. 
The rank of $H$ is $n$, that is, its dimension as a vector space over $\mathbb Z/p\Z$. 
The {\it elementary rank} of a pro-$p$ group $G$ is the supremum of all natural numbers $r$ such that $G$ has a subgroup isomorphic to an elementary $p$-group of rank $r$. 
\end{defn}

\begin{prop}\label{rank} 
Let $G$ be a cohomologically quasi-Boolean pro-$2$ group. Then the following holds:
\begin{enumerate}
\item[(i)] every involution $x\in G$ is not in the kernel of any quasi-canonical homomorphism; 
\item[(ii)] the elementary rank of $G$ is at most one. 
\end{enumerate}
\end{prop}
\begin{proof} 
To prove $(i)$, let $x\in G$ be an involution, let $H\subseteq G$ be the subgroup generated by $x$ and let $i \colon H\to G$ be the inclusion map. 
Let $i^{\bullet}\colon \Hb(G,\Z/2\Z)\to \Hb(H,\Z/2\Z)$ denote the pullback homomorphism on cohomology. 
Assume that there is a quasi-canonical homomorphism $k\in H^1(G,\mathbb Z/2\mathbb Z)$ whose restriction to $H$ is zero, or equivalently $i^{\bullet}(k)=0$. 
By Corollary \ref{negligible}, there is an $n>0$ and a $c\in H^n(G,\mathbb Z/2\mathbb Z)$ such that $i^{\bullet}(c)\in H^n(H,\mathbb Z/2\mathbb Z)$ is non-zero. 
Then, by Remark \ref{6.8}, 
\[
0\neq i^{\bullet}(c)^2=i^{\bullet}(c^2)=i^{\bullet}(c\cup k^n)=i^{\bullet}(c)\cup i^{\bullet}(k)^n=0
\]
using that $\Hb(H,\mathbb Z/2\mathbb Z)$ is isomorphic to a polynomial ring in one variable over $\mathbb Z/2\mathbb Z$, but this is a contradiction. 
Therefore $(i)$ holds. 

To prove $(ii)$, assume to the contrary that $G$ contains a subgroup $H$ isomorphic to an elementary $2$-group of rank $2$. 
Then the kernel of the restriction of a quasi-canonical homomorphism to $H$ is non-trivial. 
This contradicts part $(i)$, so claim $(ii)$ is true. 
\end{proof}

\begin{defn}\label{f-def} 
Let $p$ be a prime number and let $h \colon R\to S$ be a homomorphism of graded anti-commutative algebras over $\mathbb F_p$. 
We say that $h$ is an {\it $F$-isomorphism} if
\begin{enumerate}
\item[$(i)$] for every homogeneous element $r$ in the kernel of $h$ we have $r^n=0$ for some $n$,
\item[$(ii)$] for every homogeneous element $s$ in $S$ the power $s^{p^n}$ is in the image of $h$ for some $n$.
\end{enumerate}
\end{defn}

\begin{lemma}\label{booleup} 
Let $f \colon \Bb \to \Cb$ be an $F$-isomorphism between graded Boolean algebras. 
Then $f$ is an isomorphism. 
\end{lemma}
\begin{proof} 
Since graded Boolean algebras have no nilpotent elements, we get that $f$ is injective by condition $(i)$ of Definition \ref{f-def}. 
Next we show that $f$ is surjective. 
Since $f$ is clearly an isomorphism in degree zero, it will be sufficient to show that for every positive integer $n$ and $x\in C^n$ there is a $y\in B^n$ such that $f(y)=x$. 
By condition $(ii)$ of Definition \ref{f-def}, there is a positive integer $m$ and a $z\in B^{nm}$ such that $f(z)=x^m$. 
Since $\Bb$ is a graded Boolean algebra,  
the $m$-th power map $B^n\to B^{nm}$ is an isomorphism, 
so there is a $y\in B^n$ such that $y^m=z$, and hence $f(y)^m=f(y^m)=f(z)=x^m$. 
Since $\Cb$ is also a graded Boolean algebra, 
the $m$-th power map $C^n\to C^{nm}$ is an isomorphism, so $f(y)=x$. 
\end{proof}

\begin{notn} 
For every profinite group $G$, let $\mathfrak A(G)$ denote the set of all subgroups of $G$ which are finite elementary abelian $p$-groups. 
We note that, since $G$ is Hausdorff, all such subgroups are closed by \cite[I \S 8.2, Proposition 4 on page 77]{BourbakiTop}. 
Also note that $\mathfrak A(G)$ form a category where morphisms are maps $f \colon A\to B$ such that there is an $x\in G$ such that $f(y)=x^{-1}yx$ for every $y\in A$. 
Note that the assignment $G\mapsto\mathfrak A(G)$ is functorial, 
that is, for every continuous homomorphism $h \colon G \to H$ 
there is an induced functor $\mathfrak A(h) \colon \mathfrak A(G)\to\mathfrak A(H)$. 
For every open normal subgroup $N\triangleleft G$, let $\mathfrak A(G,N)$ denote the image of
$\mathfrak A(G)$ under the functor $\mathfrak A(\pi_N) \colon \mathfrak A(G)\to\mathfrak A(G/N)$ 
induced by the quotient map $\pi_N \colon G\to G/N$.
\end{notn}

\begin{notn} 
Let $\mathbf{DGAC}_p$ denote the category of graded anti-commutative algebras over $\mathbb Z/p\Z$. 
For every profinite group $G$, 
let $F_G \colon \mathfrak A(G)\to\mathbf{DGAC}_p$ be the functor given by the rule $A \mapsto \Hb(A,\mathbb Z/p\Z)$. 
For every open normal subgroup $N\triangleleft G$, let $F_{G,N} \colon \mathfrak A(G,N)\to\mathbf{DGAC}_p$ denote the restriction of $F_{G/N}$ onto
$\mathfrak A(G,N)$. 
For every $G$ and $N$ as above, let $\underline H^{\bullet}_{\mathfrak A}(G,\mathbb Z/p\Z)$ and $\underline H^{\bullet}_{\mathfrak A}(G,N,\mathbb Z/p\Z)$ denote the inductive limit of the functors $F_G$ and $F_{G,N}$, respectively. 
\end{notn}

\begin{notn} 
Let $G$ and $N$ be as above. 
Then for every $A\in\mathfrak A(G,N)$, 
let $\pi_{N,A} \colon A\to\pi_N(A)$ be the map induced by the restriction of $\pi_N$ onto $A$. 
We let $\pi_{N,A}^{\bullet}\colon \Hb(\pi_N(A),\Z/p\Z) \to \Hb(A,\Z/p\Z)$
denote the pullback homomorphism on cohomology.  
Note that for every
\[
\underline c=\{c_A\in \Hb(A,\mathbb Z/p\Z)\mid
A\in\mathfrak A(G,N)\}\in
\underline H^{\bullet}_{\mathfrak A}(G,N,\mathbb Z/p\Z) 
\]
the collection
\[
\mathfrak a_N(\underline c)=\{\pi_{N,A}^{\bullet}(c_{\pi_N(A)})
\in \Hb(A,\mathbb Z/p\Z)\mid
A\in\mathfrak A(G)\}
\]
lies in $\underline H^{\bullet}{\mathfrak A}(G,\mathbb Z/p\Z)$ and the map
$\mathfrak a_N \colon \underline H^{\bullet}_{\mathfrak A}(G,N,\mathbb Z/p\Z)\to\underline H^{\bullet}_{\mathfrak A}(G,\mathbb Z/p\Z)$ is a homomorphism of graded anti-commutative algebras over $\mathbb Z/p\Z$. 
Let $\Hb_{\mathfrak A}(G,\mathbb Z/p)$ denote the union of the images of 
these homomorphisms as $N$ ranges over the set of all open normal subgroups of $G$. 
Since these images form an inductive system, $\Hb_{\mathfrak A}(G,\mathbb Z/p\Z)$ is a graded subalgebra of
$\underline H^{\bullet}_{\mathfrak A}(G,\mathbb Z/p\Z)$. 
Clearly, when $G$ is finite, we have $\Hb_{\mathfrak A}(G,\mathbb Z/p\Z) = \underline H^{\bullet}_{\mathfrak A}(G,\mathbb Z/p\Z)$. 
\end{notn}

\begin{prop}\label{onlyone} 
Let $G$ be a profinite group of elementary rank at most $1$ such that there is an open normal subgroup of $G$ without $p$-torsion. 
Then $H^n_{\mathfrak A}(G,\mathbb Z/p\Z) = C(\mathcal X_p^*(G),\mathbb Z/p\Z)$ for every $n>0$ 
where $C(\mathcal X_p^*(G),\Z/p\Z)$ denotes the ring of continuous functions from $\mathcal X_p^*(G)$ to $\Z/p\Z$, 
where the latter is equipped with the discrete topology.  
\end{prop}

\begin{remark} 
Let $F(\mathcal X_p^*(G),\mathbb Z/p\Z)$ denote the group of all $\mathbb Z/p\Z$-valued functions on $\mathcal X_p^*(G)$. 
For every non-zero $A\in\mathfrak A(G)$, we have $A\cong\mathbb Z/p\Z$. 
Hence $H^n(A,\mathbb Z/p\Z)=\mathbb Z/p\Z$, every element $c\in \underline H^n_{\mathfrak A}(G,\mathbb Z/p\Z)$ gives rise to a function $\mathcal Y_p^*(G)\to\mathbb Z/p\Z$ which is conjugation-invariant, 
and hence descends to a function $f(c) \colon \mathcal X_p^*(G)\to\mathbb Z/p\Z$. 
The map $f \colon \underline H^n_{\mathfrak A}(G,\mathbb Z/p\Z)\to F(\mathcal X_p^*(G),\mathbb Z/p\Z)$ is an isomorphism. 
So the precise meaning of the claim of Proposition \ref{onlyone} 
is that the image of $H^n_{\mathfrak A}(G,\mathbb Z/p\Z)$ under $f$ and $C(\mathcal X_p^*(G),\mathbb Z/p\Z)$ are equal as subgroups of $F(\mathcal X_p^*(G),\mathbb Z/p\Z)$.  
\end{remark}

\begin{proof}[Proof of Proposition \ref{onlyone}]
First let $c\in H^n_{\mathfrak A}(G,\mathbb Z/p\Z)$ be arbitrary. 
Then there is an open normal subgroup $N\triangleleft G$ and a $d\in \underline H^n_{\mathfrak A}(G,N,\mathbb Z/p\Z)$ 
such that $c=\mathfrak a_N(d)$. 
Since for every pair $M,N$ of  open normal subgroups of $G$ 
such that $M\subseteq N$ the image of $\mathfrak a_M$ contains the image of $\mathfrak a_N$, 
we may assume that $N$ does not contain $p$-torsion without the loss of generality by shrinking $N$ if it is necessary. Then $\pi_N$ induces a map $\pi_N^{\#} \colon \mathcal X_p^*(G)\to\mathcal X_p^*(G/N)$, and $f(c)$ is the composition of $\pi_N^{\#}$ with the function $\mathrm{Im}(\pi_N^{\#})
\to\mathbb Z/p\Z$ corresponding to $d$. 
Since both $\pi_N^{\#}$ and the latter function are continuous, we get that $f(c)$ is continuous, too.

Now let $c\in C(\mathcal X_p^*(G),\mathbb Z/p\Z)$ be arbitrary, and let $g \colon \mathcal Y_p^*(G)\to\mathbb Z/p\Z$ be the continuous function we get by composing the quotient map $\mathcal Y_p^*(G)\to\mathcal X_p^*(G)$ with $c$. 
Since the translates of normal open subgroups of $G$ form a sub-basis for the topology on $G$, 
for every $x\in\mathcal Y_p^*(G)$ there is an open normal subgroup $N_x\triangleleft G$ 
such that $g$ is constant on $xN_x\cap\mathcal Y_p^*(G)$. 
Since there is an open normal subgroup of $G$ without $p$-torsion, the subspace $\mathcal Y_p^*(G)$ is closed, and hence compact, 
so there is a finite subset $S\subseteq\mathcal Y_p^*(G)$ such that $\mathcal Y_p^*(G)\subseteq
\bigcup_{x\in S}xN_x$. 

Since $N=\bigcap_{x\in S}N_x$ is the intersection of finitely many normal open subgroups, 
it is an open normal subgroup, too. 
We may even assume that $N$ does not contain $p$-torsion without the loss of generality by shrinking $N$ if it is necessary, as above. 
Then $g$ is constant on $xN\cap\mathcal Y_p^*(G)$ for every $x\in\mathcal Y_p^*(G)$, 
and hence $c$ is the composition of $\pi_N^{\#}$ above with a function $d \colon \mathrm{Im}(\pi_N^{\#})\to\mathbb Z/p\Z$.
\end{proof}

\begin{notn} 
Let $G$ be as above, and, for every $A\in\mathfrak A(G)$, let $i_A \colon A\to G$ be the inclusion map. 
For every $c\in \Hb(G,\mathbb Z/p\Z)$, the collection
$$q_G(c)=\{i_A^{\bullet}(c)\in \Hb(A,\mathbb Z/p\Z)\mid
A\in\mathfrak A(G)\}$$
lies in $\underline H^\bbb_{\mathfrak A}(G,\mathbb Z/p\Z)$ 
since every conjugation of $G$ induces the identity on $\Hb(G,\mathbb Z/p\Z)$. 
The resulting map $q_G \colon \Hb(G,\mathbb Z/p\Z)\to \underline H^\bbb_{\mathfrak A}(G,\mathbb Z/p\Z)$ is a homomorphism of graded anti-commutative algebras over $\mathbb Z/p\Z$.
\end{notn}

\begin{lemma} 
The homomorphism $q_G$ maps $\Hb(G,\mathbb Z/p\Z)$ into $\Hb_{\mathfrak A}(G,\mathbb Z/p\Z)$. 
\end{lemma}
\begin{proof} Let $c\in H^n(G,\mathbb Z/p\Z)$ be arbitrary. 
Then there is an open normal subgroup $N\triangleleft G$ and a $d\in H^n(G/N,\mathbb Z/p\Z)$ such that $c=\pi^{\bullet}_N(d)$. 
Let $\phi_N \colon \underline H^n_{\mathfrak A}(G/N,\mathbb Z/p\Z)\to\underline H^n_{\mathfrak A}(G,N,\mathbb Z/p\Z)$ be the map induced by the inclusion functor $\mathfrak A(G,N)
\to\mathfrak A(G/N)$. 
Clearly $q_G(c)= \mathfrak a_N(\phi_N(q_{G/N}(d)))$, and hence the claim holds. 
\end{proof} 

\begin{thm}[Quillen--Scheiderer]\label{quillen-profinite} 
For every profinite group $G$, the homomorphism
$$q_G \colon \Hb(G,\mathbb Z/p\Z)\to \Hb_{\mathfrak A}(G,\mathbb Z/p\Z)$$
is an $F$-isomorphism. 
\end{thm}
\begin{proof} 
This theorem was proved by Quillen when $G$ is finite (see \cite[Theorem 7.1 on page 567]{Qu} which is actually a more general result). 
At the end of \cite{Sc} (see 8.6 and 8.7 on pages 279--80) Scheiderer pointed out that there is an easy limit argument to derive the theorem above as a corollary. 
In fact, he proved that $q_G$ satisfies property $(i)$ in Definition \ref{f-def}. 
We complete his argument by showing property $(ii)$ for the convenience of the reader. 

\begin{notn} 
For every pair $M,N$ of  open normal subgroups of $G$ such that $M\subseteq N$, 
let $\pi_{M,N} \colon G/M\to G/N$ be the quotient map, and let $\mathfrak A(G,M,N)$ denote the image of the functor
$\mathfrak A(\pi_{M,N}) \colon \mathfrak A(G/M)\to\mathfrak A(G/N)$.
\end{notn}

\begin{lemma}\label{schei} 
For every open normal subgroup $N\triangleleft G$, there is open normal subgroup $M\triangleleft G$ such that $M\subseteq N$ and $\mathfrak A(G,M,N)=\mathfrak A(G,N)$.  
\end{lemma}
\begin{proof} 
This is the Claim on page 280 of \cite{Sc}.
\end{proof}

Now let $c\in H^n_{\mathfrak A}(G,\mathbb Z/p\Z)$ be arbitrary, and let $N\triangleleft G$ be an open normal subgroup such that there is a
$$\underline d=\{d_A\in \Hb(A,\mathbb Z/p\Z)\mid
A\in\mathfrak A(G,N)\}\in
\underline H^n_{\mathfrak A}(G,N,\mathbb Z/p\Z)$$
with the property $c=\mathfrak a_N(\underline d)$. 
By Lemma \ref{schei}, there is an open normal subgroup $M\triangleleft G$ such that $M\subseteq N$ and $\mathfrak A(G,M,N)=\mathfrak A(G,N)$.  
For every $A\in\mathfrak A(G/M)$,  let
$\pi_{M,N,A} \colon A\to\pi_{M,N}(A)$ be the map induced by the restriction of
$\pi_{M,N}$ onto $A$. 
The fact $\mathfrak A(G,M,N)=
\mathfrak A(G,N)$ means that the collection
$$e = \left\{\pi^{\bullet}_{M,N,A}(d_{\pi_{M,N}(A)})\in \Hb(A,\mathbb Z/p\Z)\mid
A\in\mathfrak A(G/M)\right\}$$
is a well-defined element of $H^n_{\mathfrak A}(G/M,\mathbb Z/p\Z)=
\underline H^n_{\mathfrak A}(G/M,\mathbb Z/p\Z)$. 
Applying Quillen's theorem to $e$ we get that there is a positive integer $m$ and an $f\in H^n(G/M,\mathbb Z/p\Z)$ such that $e^{p^m}=q_{G/M}(f)$. 
Let $\phi_M \colon \underline H^n_{\mathfrak A}(G/M,\mathbb Z/p\Z)\to\underline H^n_{\mathfrak A}(G,M,\mathbb Z/p\Z)$ be the map induced by the inclusion functor $\mathfrak A(G,M)\to\mathfrak A(G/M)$ as above. 
Then $\mathfrak a_M(\phi_M(e))=\mathfrak a_N(\underline d)$, and hence
\[
c^{p^m}=
\mathfrak a_M(\phi_M(e))^{p^m}=\mathfrak a_M(\phi_M(e^{p^m}))=
\mathfrak a_M(\phi_M(q_{G/M}(f)))=q_G(\pi_N^*(f)). 
\]
This finishes the proof of Theorem \ref{quillen-profinite}. 
\end{proof}

As a consequence we get the following 
\begin{cor}\label{local-global} 
Let $G$ be a cohomologically quasi-Boolean pro-$2$ group. 
Then for every $i>0$, there is a natural homomorphism 
$$\pi_i \colon H^i(G,\mathbb Z/2\mathbb Z) \to C(\mathcal X^*(G),\mathbb Z/2\mathbb Z)$$
which is an isomorphism for $i > 1$ and surjective for $i=1$. 
\end{cor}
\begin{proof} 
Let $\Cb(\mathcal X^*(G),\mathbb Z/2\mathbb Z)$ denote the graded Boolean algebra associated to the Boolean ring $C(\mathcal X^*(G),\mathbb Z/2\mathbb Z)$. 
By Proposition \ref{rank}, the conditions of Proposition \ref{onlyone} apply to $G$. 
Hence, by Theorem \ref{quillen-profinite}, there is an $F$-isomorphism:
$$\pi_* \colon \Hb(G,\mathbb Z/2\mathbb Z)\to
\Cb(\mathcal X^*(G),\mathbb Z/2\mathbb Z).$$
This map is zero on the nilradical of $\Hb(G,\mathbb Z/2\mathbb Z)$, so it induces an $F$-isomorphism:
\begin{equation}
\Hb(G,\mathbb Z/2\mathbb Z)/\mathcal N(\Hb(G,\mathbb Z/2\mathbb Z))\to
\Cb(\mathcal X^*(G),\mathbb Z/2\mathbb Z)\label{7.17.1}
\end{equation}
of graded Boolean algebras by Lemma \ref{7.6}, which must be an isomorphism by Lemma \ref{booleup}. 
Since the nilradical of $\Hb(G,\mathbb Z/2\mathbb Z)$ consists of degree one elements, the claim follows. 
\end{proof}


\section{Proof of the main theorem and some consequences}\label{sec:main_thm_and_consequences}

Now we are ready to give the proofs of our main results. 

\begin{proof}[Proof of Theorem \ref{bigbigbig}] 
The implication $(i)\Rightarrow(ii)$ is Theorem \ref{qb_rp}, while the implication $(ii)\Rightarrow(i)$ is Theorem \ref{rp_qb}. 
We already saw that $(ii)$ trivially implies $(iii)$ in the proof of Theorem \ref{foximaxi}. 
Now assume that $G$ is the maximal pro-$2$ quotient of a real projective profinite group $H$. 
As explained in \cite[Section 10]{PQ}, $\Hb(H,\mathbb Z/2\mathbb Z)$ is the connected sum of a dual algebra and a graded Boolean algebra. 
The pull-back map $\Hb(H,\mathbb Z/2\mathbb Z)\to \Hb(G,\mathbb Z/2\mathbb Z)$ is an isomorphism by the Rost--Voevodsky norm residue theorem \cite{Voe}, 
so we get that $G$ is cohomologically quasi-Boolean. 
Therefore the implication $(iii)\Rightarrow(iv)$ holds. 
To finish the proof of Theorem \ref{bigbigbig}, we show the remaining implication $(iv)\Rightarrow(i)$ in the following

\begin{thm}\label{thm:coh_quasi_is_rp} 
Every cohomologically quasi-Boolean pro-$2$ group $G$ is real projective. 
\end{thm}
\begin{proof} 
By part $(i)$ of Proposition \ref{rank}, $G$ has an open subgroup without $2$-torsion. 
By Remark \ref{rem:equiv_of_defns}, 
it will therefore be sufficient to show that every real embedding problem for $G$ has a solution. 
By Proposition \ref{real2real}, we need to show that any real $2$-embedding problem for $G$ has a solution. 
In fact we will prove something stronger.

\begin{defn}\label{deflif} 
For every group $G$, let $G^{\#}$ denote the set of its conjugacy classes.  
For  every $x \in G$, let $x^{\#}\in G^{\#}$ denote the conjugacy class of $x$, 
and, for every homomorphism $h \colon G\to H$ of groups, 
let $h^{\#} \colon G^{\#}\to H^{\#}$ denote the map on conjugacy classes induced by $h$. 
Now let $G$ be a profinite group. 
An {\it embedding problem with  lifting data $(\mathbf E,f)$} for $G$ is an embedding problem $\mathbf E$:
\[
\xymatrix{
 & G \ar[d]^{\phi}\ar@{.>}[ld]_{\widetilde{\phi}}
\\
B\ar[r]_{\alpha} & A}
\]
and a continuous map $f \colon \mathcal X^*(G)\to\mathcal X(B)$ such that
$\alpha^{\#}\circ f=\phi^{\#}|_{\mathcal X^*(G)}$. 
A {\it solution} to this embedding problem with lifting data is a solution $\widetilde{\phi}$ to the embedding problem $\mathbf E$ such that $\widetilde{\phi}^{\#}|_{\mathcal X^*(G)}=f$. 
\end{defn}

\begin{prop}\label{lifting} 
Let $G$ be a quasi-Boolean pro-$2$ group $G$. Then every $2$-embedding problem with lifting data for $G$ has a solution. 
\end{prop}
\begin{proof} In order to prove the claim in a first significant case, we need to recall some basic definitions and results. 

\begin{defn} 
The {\it kernel}, denoted $\mathrm{Ker}(\mathbf E)$, of an embedding problem $\mathbf E$ as one in Definition \ref{4.1} is the kernel of $\alpha$. 
We say that $\mathbf E$ is {\it central} if $\mathrm{Ker}(\mathbf E)$ lies in the centre of $B$. 
In this case the conjugation action of $G$ makes $\mathrm{Ker}(\mathbf E)$ into a constant abelian $G$-module. 
Assume now that the embedding problem $\mathbf E$ is central. 
Let $\widehat{\phi} \colon G\to B$ be a continuous map such that $\alpha\circ
\widehat{\phi}=\phi$. 
Then the map $c \colon G\times G\to\mathrm{Ker}(\mathbf E)$ given by the rule:
$$c(x,y)=\widehat{\phi}(xy)\widehat{\phi}(y)^{-1}\widehat{\phi}(x)^{-1}\in
\mathrm{Ker}(\mathbf E),\quad(x,y\in G)$$
is a cocycle, and its cohomology class $o(\mathbf E)\in H^2(G,\mathrm{Ker}(\mathbf E))$, called the {\it obstruction class of $\mathbf E$}, 
does not depend on the choice of $\widehat{\phi}$, only on $\mathbf E$. 
Moreover, $\mathbf E$ has a solution if and only if $o(\mathbf E)$ is zero. 
\end{defn}

\begin{rem}\label{onatural} 
The obstruction class has the following important naturality property: 
Let $\mathbf E$ be an embedding problem for $G$ as above, and suppose that $\mathbf E$ is central. 
Let $\chi \colon H\to G$ be a continuous homomorphism of profinite groups. Then
$$\xymatrix{
 & H \ar[d]^{\phi\circ\chi}\ar@{.>}[ld]_{\widetilde{\psi}}
\\
B\ar[r]_{\alpha} & A}$$
is a central embedding problem $\mathbf E(\chi)$ for $H$ with the same kernel as $\mathbf E$, and we have
\[
\chi^{\bullet}(o(\mathbf E))=o(\mathbf E(\chi)),
\]
where $\chi^{\bullet} \colon \Hb(G,\mathrm{Ker}(\mathbf E))\to \Hb(H,\mathrm{Ker}(\mathbf E))$ 
is the pull-back map on cohomology.
\end{rem}

\begin{lemma}\label{2kernel} Let $G$ be a quasi-Boolean pro-$2$ group $G$. Then every $2$-embedding problem with lifting data for $G$ and with a kernel isomorphic to $\mathbb Z/2\mathbb Z$ has a solution. 
\end{lemma}
\begin{proof} Let $(\mathbf E,f)$ be an embedding problem with lifting data for $G$ as in Definition \ref{deflif}, and assume that its kernel is isomorphic to $\mathbb Z/2\mathbb Z$. 
Since the automorphism group of the latter is trivial, we get that $\mathbf E$ is central. 
Because $\mathbf E$ is equipped with lifting data, it is real, that is, for every subgroup $H\subseteq G$ of order $2$ the embedding problem $\mathbf E(i_H)$ has a solution, where $i_H \colon H\to G$ is the inclusion map. 
Therefore, by Remark \ref{onatural}, the image of $o(\mathbf E)$ under the homomorphism:
$$\pi_2 \colon H^2(G,\mathbb Z/2\mathbb Z) \to C(\mathcal X^*(G),\mathbb Z/2\mathbb Z)$$
is zero. 
So by Corollary \ref{local-global}, the obstruction class $o(\mathbf E)$ vanishes, and hence $\mathbf E$ has a solution.

Let $s \colon G\to B$ be such a solution. 
Let $r \colon \mathcal X^*(G)\to\mathbb Z/2\mathbb Z = \mathrm{Ker}(\alpha)$ be the map given by the rule:
$$r(x)=\begin{cases}
    0 & \text{, if } s^{\#}(x)=f(x),\\
    1 & \text{, otherwise. } \end{cases}$$
Since $s^{\#}\times f \colon \mathcal X^*(G)\to B^{\#}\times B^{\#}$ is continuous with finite image, and $r(x)$ only depends on $s^{\#}(x)$ and $f(x)$ for every $x\in\mathcal X^*(G)$, we get that $r$ is continuous. 
Therefore, by Corollary \ref{local-global}, there is a continuous homomorphism
$\chi \colon G\to\mathbb Z/2\mathbb Z = \mathrm{Ker}(\alpha)$ whose image under the homomorphism  
\[
\pi_1 \colon H^1(G,\mathbb Z/2\mathbb Z) \to C(\mathcal X^*(G),\mathbb Z/2\mathbb Z)
\]
is $r$. 
Let $\widetilde{\phi} \colon G\to B$ be the function given by the rule
$\widetilde{\phi}(g)=s(g)\chi(g)$. 
Since it is the product of two continuous functions, $\widetilde{\phi}$ is continuous. 
Moreover, 
\[
\widetilde{\phi}(gh)=s(gh)\chi(gh)=s(g)s(h)\chi(g)\chi(h)=
s(g)\chi(g)s(h)\chi(h)=
\widetilde{\phi}(g)\widetilde{\phi}(h)
\]
using that $s$ and $\chi$ are homomorphisms and $\mathrm{Ker}(\alpha)$ is central. 
Therefore $\widetilde{\phi}$ is a homomorphism. 
Since $\alpha \circ\widetilde{\phi}=\alpha\circ s=\phi$, we get that $\widetilde{\phi}$ is a solution to $\mathbf E$. 
Now let $y\in\mathcal Y^*(G)$ be arbitrary. 
If  $s^{\#}(y^{\#})=f(y^{\#})$, then $\widetilde{\phi}(y)=s(y)$, and hence
$\widetilde{\phi}(y)^{\#}=s(y)^{\#}=s^{\#}(y^{\#})=f(y^{\#})$. 
If $s^{\#}(y^{\#})\neq f(y^{\#})$, then $\widetilde{\phi}(y)\neq s(y)$, so
$\widetilde{\phi}(y)$ is the unique element of $\alpha^{-1}(\phi(y))$ distinct from $s(y)$. 
Since $\alpha^{-1}(\phi(y))$ contains an element of $f(y^{\#})$, as $\alpha$ is surjective, and this element is not $s(y)$, it must be $\widetilde{\phi}(y)$. So $\widetilde{\phi}(y)^{\#}=f(y^{\#})$ in this case, too.
\end{proof}

Now we are going to show Proposition \ref{lifting} in the general case. Let $(\mathbf E,f)$ be an embedding problem with lifting data for $G$ as in Definition \ref{deflif}. 
Since $B$ is a $2$-group, it has a filtration by normal subgroups:
$$\{1\}=N_0\subset N_1\subset\cdots\subset N_n=\mathrm{Ker}(\alpha)$$
such that the kernel of the quotient map $\pi_k \colon B/N_k\to B/N_{k+1}$ is isomorphic to $\mathbb Z/2\mathbb Z$ for every $k=0,1,\ldots,n-1$. 
Let $q_k \colon B\to B/N_k$ be the quotient map. 
Note that it will be sufficient to show that for every continuous homomorphism $h \colon G\to B/N_{k+1}$ such that
$h^{\#}|_{\mathcal X^*(G)}=q_{k+1}^{\#}\circ f=\pi_k^{\#}\circ q_k^{\#}\circ f$ the embedding problem $\mathbf E_k$:
\[
\xymatrix{
 & G \ar[d]^-{h}\ar@{.>}[ld]_-{\widetilde h}
\\
B/N_k\ar[r]_-{\!\!\pi_k} &
B/N_{k+1}}
\]
with lifting data $q_k^{\#}\circ f$ has a solution for every $k=0,1,\ldots,n$. 
Indeed let $r_k \colon B/N_k\to A=B/N_n$ be the quotient map.  
Then we would get by descending induction on the index $k$ that  the embedding problem:
\[
\xymatrix{
 & G \ar[d]^-{\phi}\ar@{.>}[ld]_-{\widetilde{\phi}}
\\
B/N_k\ar[r]_-{\ \ r_k} & A}
\]
with lifting data $q_k^{\#}\circ f$ has a solution. The claim is now clear from the case $k=0$. 
However, $\mathbf E_k$ has a kernel isomorphic to $\mathbb Z/2\mathbb Z$, so $(\mathbf E_k,q_k^{\#}\circ f)$ has a solution by Lemma \ref{2kernel}. 
\end{proof}
In order to conclude the proof of Theorem \ref{thm:coh_quasi_is_rp} it will be sufficient to show that every real $2$-embedding problem $\mathbf E$ for $G$ as above can be equipped with lifting data. 
By assumption, for every $x\in\mathcal X(A)$ in the image $\mathrm{Im}(\phi^{\#}|_{\mathcal X^*(G)})$, there is a $y\in\mathcal X(B)$ such that
$\alpha^{\#}(y)=x$, 
i.e., there is a section $g \colon \mathrm{Im}(\phi^{\#}|_{\mathcal X^*(G)})\to\mathcal X(B)$ of the restriction $\alpha^{\#}|_{\mathcal X(B)}$. 
Since $g$ is a map between discrete spaces, it is continuous, therefore the composition $f=g\circ\phi^{\#}|_{\mathcal X^*(G)}$ is also continuous, and hence
$(\mathbf E,f)$ is a $2$-embedding problem with lifting data. 
\end{proof}
\renewcommand{\qedsymbol}{}
\end{proof}


\begin{proof}[Proof of Corollary \ref{justboole}] 
First assume that $G$ is Boolean, i.e., $G$ is isomorphic to $\mathbb B(X)$ for some profinite space $X$. 
Then it is cohomologically quasi-Boolean by Theorem \ref{bigbigbig}, so $\Hb(G,\mathbb Z/2\mathbb Z)$ is the connected sum of a dual algebra $\Db$ and a graded Boolean algebra $\Bb$. 
Assume that $\Db$ is non-trivial, so there is a non-zero $c\in D^1\subset\textrm{Hom}(G,\mathbb Z/2\mathbb Z)$. 
Then $c\cup c=0$ as $D^2=0$, so the restriction of $c$ onto every involution in the image of $i_X \colon X\to\mathcal Y(\mathbb B(X))$ is zero. 
Therefore, by the universal property of $\mathbb B(X)$ of Remark \ref{uki-ho}, 
the homomorphism corresponding to $c$ is also zero, a contradiction. 
Hence $G$ is cohomologically Boolean.

Next assume that $G$ is cohomologically Boolean. 
Then it is quasi-Boolean by Theorem \ref{bigbigbig}, 
so $G$ is the free product of a free pro-$2$ group $F$ and a Boolean group $\mathbb B(X)$ for a profinite space $X$. 
Assume that $F$ is non-trivial. 
Then there is a non-zero $c\in\textrm{Hom}(F,\mathbb Z/2\mathbb Z)$. 
Let $\pi_1 \colon G\to F$ be the surjective homomorphism supplied by the universal property of free pro-$2$ products, and let $\overline c$ be the composition of $\pi_1$ and $c$.
Then $c\cup c=0$, and hence $\overline c\cup\overline c=0$. 
Since $\Hb(G,\mathbb Z/2\mathbb Z)$ is a graded Boolean algebra, this implies that $\overline c$ is zero, a contradiction. 
Therefore $G$ is Boolean.
\end{proof}

\begin{rem} 
Let $G$ be a Boolean group and let $f \colon G\to\mathbb Z/2\mathbb Z$ be a quasi-canonical homomorphism. 
By Remark \ref{6.8}, the homomorphism $f$ is characterised by the property 
that, for every $n>0$ and $c\in H^n(G,\mathbb Z/2\mathbb Z)$, we have $c^2=c\cup k^n$. 
But $\Hb(G,\mathbb Z/2\mathbb Z)$ is a graded Boolean algebra by Corollary \ref{justboole}, and hence $f$ is unique. 
Therefore, it is justified to call it the {\it canonical homomorphism} of the Boolean group $G$. 
\end{rem}

\begin{proof}[Proof of Theorem \ref{reconstruct}] 
By assumption, $G$ is isomorphic to $F(Y')*_2\mathbb B(X')$ for a set $Y'$ and a profinite space $X'$. 
We are actually going to show a stronger claim than in the theorem, 
namely, that the profinite spaces $X$ and $X'$ are homeomorphic, and the sets $Y'$ and $Y$ are bijective. 
By Theorem \ref{sistar}, the profinite spaces $X'$ and $\mathcal X^*(G)$ are homeomorphic, 
and, as we saw in the proof of Corollary \ref{local-global}, the graded Boolean algebras $\Bb$ and $\Cb(\mathcal X^*(G),\mathbb Z/2\mathbb Z)$ are isomorphic (see \eqref{7.17.1}). 
Therefore, by Stone duality of Theorem \ref{duality}, 
the profinite spaces $X$ and $X'$ are homeomorphic, too. 
Using Notation \ref{notn:lower_star}, we have $G_*=F(Y')_*\oplus\mathbb B(X')_*$. 
Hence we have $H^1(G,\mathbb Z/2\mathbb Z)=H^1(F(Y'),\mathbb Z/2\mathbb Z) \oplus H^1(\mathbb B(X'),\mathbb Z/2\mathbb Z)$. 
However, since $H^1(\mathbb B(X'),\mathbb Z/2\mathbb Z)\cong B^1$ by the above, we get that $H^1(F(Y'),\mathbb Z/2\mathbb Z)\cong D^1$, 
and hence $H^1(G,\mathbb Z/2\mathbb Z)=D^1 \oplus B^1$.   
Therefore, $\mathbb Z/2\mathbb Z^{\oplus Y'}\cong\mathbb Z/2\mathbb Z^{\oplus Y}$, 
and hence there is a bijection between the sets $Y'$ and $Y$. 
\end{proof}


\end{document}